\documentclass[a4paper,twoside,11pt]{article}

\usepackage{a4wide, fancyhdr, amsmath, amssymb, mathtools, yfonts}
\usepackage{mathrsfs}
\usepackage{graphicx}
\usepackage{tikz}
\usepackage[all]{xy}
\usepackage[utf8]{inputenc}
\usepackage{amsthm}
\usepackage[english]{babel}
\usepackage{chngcntr}
\usepackage{ifthen}
\usepackage{calc}
\usepackage{hyperref}
\numberwithin{equation}{section}

\renewcommand\labelenumi{(\roman{enumi})}
\renewcommand\theenumi\labelenumi

\newcommand{\N}{\mathbb{N}}
\newcommand{\Z}{\mathbb{Z}}
\newcommand{\Q}{\mathbb{Q}}

\newcommand\FF{\mathbb{F}}

\newcommand\DD{\mathcal{D}}

\newcommand\B{\mathfrak{B}}

\newcommand{\leg}[2]{\left(\frac{#1}{#2}\right)}
\newcommand{\M}{\mathcal{M}}
\newcommand\EE{\mathbb{E}}

\DeclareMathOperator{\Gal}{Gal}

\DeclareMathOperator{\rk}{rk}
\DeclareMathOperator{\li}{li}
\DeclareMathOperator{\Ei}{Ei}
\DeclareMathOperator{\cork}{corank}
\DeclareMathOperator{\prob}{Prob}
\DeclareMathOperator{\Msym}{Sym}
\DeclareMathOperator{\Mat}{Mat}
\DeclareMathOperator{\Art}{Art}
\DeclareMathOperator{\Frob}{Frob}
\DeclareMathOperator{\im}{im}
\DeclareMathOperator{\CL}{Cl}

\newtheorem{lemma}{Lemma}[section]
\newtheorem{theorem}[lemma]{Theorem}
\newtheorem{prop}[lemma]{Proposition}

\newtheorem{mydef}[lemma]{Definition}

\title{
\vspace{-\baselineskip}\sffamily\bfseries On the negative Pell equation}
\author{Stephanie Chan\thanks{Department of Mathematics, University College London, Gower Street, London, WC1E 6BT, United Kingdom, \url{stephanie.chan.16@ucl.ac.uk}}
,
Peter Koymans\thanks{Mathematisch Instituut, Leiden University, Niels Bohrweg 1, 2333 CA Leiden, Netherlands, \url{p.h.koymans@math.leidenuniv.nl}}
,
Djordjo Milovic\thanks{Department of Mathematics, University College London, Gower Street, London, WC1E 6BT, United Kingdom, \url{djordjo.milovic@ucl.ac.uk}}
, and
Carlo Pagano\thanks{Max Planck Institute for Mathematics, Vivatsgasse 7, 5311  Bonn, Germany, \url{carlein90@gmail.com}}}

\date{\today}

\begin{document}
\maketitle

\begin{abstract}
Using a recent breakthrough of Smith \cite{Smith}, we improve the results of Fouvry and Kl\"uners \cite{FK1} on the solubility of the negative Pell equation.
Let $\mathcal{D}$ denote the set of fundamental discriminants having no prime factors congruent to $3$ modulo $4$. Stevenhagen \cite{Stevenhagen} conjectured that the density of $D$ in $\mathcal{D}$ such that the negative Pell equation $x^2-Dy^2=-1$ is solvable with $x,y\in\Z$ is $58.1\%$, to the nearest tenth of a percent. By studying the distribution of the $8$-rank of narrow class groups $\CL^+(D)$ of $\Q(\sqrt{D})$, we prove that the infimum of this density is at least $53.8\%$.
\end{abstract}

\tableofcontents

\section{Introduction}
\label{sIntro}
In recent years, much progress has been made in the study of the distribution of $2$-parts of class groups of quadratic number fields, most notably by Fouvry and Kl\"{u}ners \cite{FK1} and Smith \cite{Smith}. One way to test the robustness of new methods in this subject is to study their applications to a conjecture of Stevenhagen \cite{Stevenhagen} concerning the solvability over $\Z$ of the negative Pell equation
\begin{equation}\label{negativePell}
x^2 - Dy^2 = -1.
\end{equation}
Here and henceforth we take $D$ to be a positive fundamental discriminant. The equation~\eqref{negativePell} is solvable over $\Z$ if and only if the ordinary and narrow class groups of the quadratic field $\Q(\sqrt{D})$, denoted by $\CL(D)$ and $\CL^+(D)$ respectively, coincide. As the odd parts of $\CL(D)$ and $\CL^+(D)$ are isomorphic, the frequency of solvability of~\eqref{negativePell} is intricately related to the joint distribution of $2$-primary parts $\CL(D)$ and $\CL^+(D)$. We note that $\CL(D)/2\CL(D) \cong \CL^+(D)/2\CL^+(D)$ if and only if $D$ is in the set
\[
\mathcal{D} = \{D\text{ positive fundamental discriminant}: p\not\equiv 3\bmod 4\text{ for all primes }p\mid D\},
\]
which we occasionally refer to as the \textit{Pell family} of fundamental discriminants. As $\mathcal{D}$ has natural density $0$ in the set of all positive fundamental discriminants, it is more meaningful to study density questions concerning the solvability of~\eqref{negativePell} relative to $\mathcal{D}$ than relative to the set of all positive fundamental discriminants. 

One of the main conjectures in \cite{Stevenhagen} is that
\begin{equation}\label{fullconj}
\lim_{X\rightarrow\infty}\frac{|\mathcal{D}^{-}(X)|}{|\mathcal{D}(X)|} = 1-\alpha = 0.58057...,
\end{equation} 
where
$$
\mathcal{D}(X) = \{D\in\mathcal{D}: D\leq X\},
$$
$$
\mathcal{D}^{-}(X) = \{D\in\mathcal{D}(X): \eqref{negativePell}\text{ is solvable over }\Z\},
$$
and
\[
\alpha =  \prod_{j \text{ odd}}(1 - 2^{-j}) = \prod_{j = 1}^{\infty}(1 + 2^{-j})^{-1} = 0.41942\ldots.
\] 
Until now, the best bounds in the direction of Stevenhagen's conjecture are due to Fouvry and Kl\"{u}ners \cite{FK2, FK3}, who used the methods they developed in \cite{FK1} to prove that 
\begin{equation}\label{FKbounds}
\frac{5}{4}\alpha \leq \liminf_{X\rightarrow\infty}\frac{|\mathcal{D}^{-}(X)|}{|\mathcal{D}(X)|}\leq \limsup_{X\rightarrow\infty}\frac{|\mathcal{D}^{-}(X)|}{|\mathcal{D}(X)|}\leq \frac{2}{3}.
\end{equation}
By incorporating the methods developed by Smith \cite{Smith}, we can improve the lower bound. 
\begin{theorem}\label{Thm1}
With $\mathcal{D}(X)$, $\mathcal{D}^-(X)$, and $\alpha$ defined as above, we have
\[
\liminf_{X\rightarrow\infty}\frac{|\mathcal{D}^{-}(X)|}{|\mathcal{D}(X)|}\geq \alpha\beta,
\]
where
\[
\beta = \sum_{n = 0}^{\infty}2^{-n(n+3)/2} = 1.28325\ldots.
\]
\end{theorem}
We note that $\beta>5/4$. To the nearest tenth of a percent, Stevenhagen's conjecture states that the density of $D\in\DD$ for which \eqref{negativePell} is solvable over $\Z$ is $58.1\%$, Fouvry and Kl\"{u}ners proved that the lower density is at least $52.4\%$, and we prove that the lower density is at least $53.8\%$.

For a finite abelian group $G$ and an integer $k\geq 1$, we let $\rk_{2^k}G = \dim_{\FF_2}(2^{k-1}G/2^kG)$; this is called the $2^k$-rank of $G$. The non-increasing sequence of non-negative integers $\{\rk_{2^k}G\}_k$ determines the isomorphism class of the $2$-primary part of $G$. Hence
\[
\eqref{negativePell}\text{ is solvable}\Longleftrightarrow \rk_{2^k}\CL(D) = \rk_{2^k}\CL^+(D)\text{ for all integers }k\geq 1.
\]
The lower bound in \eqref{FKbounds} comes from proving that the density of $D\in\mathcal{D}$ such that
\[
\rk_4\CL^+(D) = 0
\]
is equal to $\alpha$ and the density of $D\in\mathcal{D}$ such that 
\[
\rk_4\CL(D) = \rk_4\CL^+(D) = 1\text{ and }\rk_8\CL^+(D) = 0
\]
is equal to $\alpha/4$. We obtain our lower bound by proving that the density of $D\in\mathcal{D}$ such that 
\[
\rk_4\CL(D) = \rk_4\CL^+(D) = n\text{ and }\rk_8\CL^+(D) = 0.
\]
is equal to $2^{-n(n+3)/4}\alpha$. In fact, we will prove more.
\begin{theorem}\label{Thm2}
Let $\mathcal{D}(X)$ and $\alpha$ be as above, and, for integers $n\geq m\geq 0$, let
\[
\mathcal{D}_{n, m}(X) = \{D\in \mathcal{D}(X): \rk_4\CL(D) = \rk_4\CL^+(D) = n\text{ and }\rk_8\CL^+(D) = m\}
\]
Then
\[
\lim_{X\rightarrow\infty}\frac{|\mathcal{D}_{n, m}(X)|}{|\mathcal{D}(X)|} = \alpha\cdot 2^{-n(n+1)}\frac{\prod_{j = m+1}^n(2^n - 2^{n-j})}{\prod_{k = 1}^m(2^k-1)\prod_{l = 1}^{n-m}(2^l-1)}.
\]
\end{theorem}
We note that our proof of Theorem~\ref{Thm2} gives an alternative proof of \cite[Theorem~2]{FK1} and \cite[Theorem~2]{FK3}.

The major novel difficulty with working in the Pell family is that the discriminants $D \in \mathcal{D}$ have the remarkable property that the spaces 
\[
\{a \mid D : a > 0,\ a \text{ squarefree }, (a, -D/a) = 1\}
\]
and
\[
\{b \mid D : b > 0,\ b \text{ squarefree }, (b, D/b) = 1\}
\]
coincide. However, for Smith's method to work, it is essential that these spaces are typically disjoint. For instance, this is used in \cite[p.76]{Smith} to argue that most assignments $a$ are generic. If $a$ is not generic, then $a$ ends up in the error term. The reason for this is that the algebraic results break down in this case since there is no valid choice of ``variable indices''. In particular, all discriminants $D \in \mathcal{D}$ end up in the error term of Smith's theorem. It is therefore of utmost importance to extend Smith's algebraic results.

We introduce a more careful notion of genericity in equation~\eqref{eGenVar} and equation~\eqref{eGenCheb} to circumvent this pitfall. We have also devoted Section \ref{sAlg} to prove several new algebraic results. These algebraic results essentially rely on the fact that we are working with the 8-rank, which brings manipulations with R\'{e}dei symbols into play, see \cite{Redei-Stevenhagen} for an extensive treatment of R\'{e}dei symbols. Note that this approach is inspired by Smith's first paper \cite{Smith8}. However, the result in \cite{Smith8} assumes GRH, which we avoid by borrowing from the ideas that Smith introduced in his breakthrough paper \cite{Smith}.

In Section~\ref{sPrimes}, we give more direct proofs of the results that appear in \cite[Section 5]{Smith} and concern the typical distribution of prime divisors of a squarefree integer. Of course, we once again adapt these results to $D$ coming from the Pell family $\DD$. 

Finally, we would like to mention a recent paper of Knight and Xiao \cite{KX} claiming to establish \eqref{fullconj} in full. However, we were unable to verify \cite[Equation (9.8)]{KX}, which is related to the issues of genericity discussed above.

\subsection*{Acknowledgements}
We would like to thank Andrew Granville for his support throughout this project.
We wish to thank Peter Stevenhagen for some very helpful discussions at the beginning of the project. We are grateful to Alexander Smith for answering questions about his paper and for several conversations on the challenges of this project. We would also like to thank Kevin Ford, Andrew Granville and Sam Porritt for their useful suggestions in writing some of the proofs in Section~4.

The first and third authors were supported by the European Research Council grant agreement No.\ 670239.
The fourth author is grateful to the Max Planck Institute for Mathematics of Bonn for its financial support.

\section{Algebraic results}
\label{sAlg}
We start this section by introducing the R\'{e}dei symbol, which will play a prominent role throughout the paper. Then we prove several identities on the sum of four R\'{e}dei symbols, which serve as the algebraic input for our analytic machinery in proving equidistribution.

\subsection{R\'{e}dei symbols}
We shall review the fundamental properties of the R\'{e}dei symbols, needed to state and establish Theorem~\ref{t2.8}, Theorem~\ref{t2.8self}, Theorem~\ref{t.2.8swapmin} and Theorem~\ref{t.2.8swapped}. Our main reference is Stevenhagen's recent work \cite{Redei-Stevenhagen}.

\begin{mydef}
Write $\Omega$ for the collection of the places of $\Q$. 
For a place $v$ in $\Omega$, we write $(-,-)_v$ for the Hilbert symbol. 
If $K/\Q$ is a finite extension, write $\Delta_K$ as the discriminant of $K/\Q$.
\end{mydef}
\begin{mydef}
Let $a,b \in \Q^{*}/(\Q^{*})^2$.
If $a$ is non-trivial, write $\chi_a$ as the unique character $\chi_a: G_{\Q}\rightarrow \mathbb{F}_2$ with kernel $G_{\Q(\sqrt{a})}$.
We say that $(a,b)$ is acceptable if we have that $(a,b)_v=1$ for each $v \in \Omega$.  
\end{mydef}

In case one of $a,b$ is trivial, then $(a,b)$ is clearly acceptable. Now suppose $a$ and $b$ are both non-trivial. Then $(a,b)$ is acceptable if and only if there exists a Galois extension $L/\Q$ containing $\Q(\sqrt{a},\sqrt{b})$, with $\Gal(L/\Q(\sqrt{ab}))$ cyclic of order $4$, and such that every element $\sigma \in \Gal(L/\Q)$ with $\chi_a(\sigma)\neq \chi_b(\sigma)$ must be an involution, i.e. $\sigma^2=\text{id}$.  If $a=b$, we are simply requiring $L/\Q$ to be a cyclic extension of degree $4$ of $\Q$ containing $\Q(\sqrt{a})$. While if $a \neq b$, we are requiring $L/\Q$ to be dihedral of degree $8$, with $\Gal(L/\Q(\sqrt{ab}))$ cyclic of order $4$. 
When $a,b$ are both non-trivial and $(a,b)$ is acceptable, denote by $\mathcal{F}_{a,b}$ the collection of fields $L/\Q$ described above. 

Write $\Gamma_{\mathbb{F}_2}(\Q):=\text{Hom}_{\text{top.gr.}}(G_{\Q},\mathbb{F}_2)$. For $\chi \in \Gamma_{\mathbb{F}_2}(\Q)$, write $\Q(\chi):=(\Q^{\text{sep}})^{\text{ker}(\chi)}$. We put $\Gamma_{\mathbb{F}_2}(\Q,\{a,b\}):=\frac{\Gamma_{\mathbb{F}_2}(\Q)}{\langle \{\chi_a,\chi_b\} \rangle}$. It can be easily shown that the set $\mathcal{F}_{a,b}$ is equipped of a \emph{difference}, which is a map
\[-:\mathcal{F}_{a,b} \times \mathcal{F}_{a,b} \to \Gamma_{\mathbb{F}_2}(\Q,\{a,b\}),
\] 
with the property that for $L_1,L_2 \in \mathcal{F}_{a,b}$ one has  $\Q(\chi)\cdot L_2 \supseteq L_1$ if and only if $\chi=L_2-L_1$. For $L_1,L_2,L_3 \in \mathcal{F}_{a,b}$  one has that
\[(L_3-L_2)+(L_2-L_1)=L_3-L_1.
\]
We have that $L_2-L_1=0$ if and only if $L_1=L_2$. Therefore each $L \in \mathcal{F}_{a,b}$ induces an explicit bijection between $\mathcal{F}_{a,b}$ and $\Gamma_{\mathbb{F}_2}(\Q,\{a,b\})$. For any subgroup $H \leq \Gamma_{\mathbb{F}_2}(\Q,\{a,b\})$, we say that $S \subseteq \mathcal{F}_{a,b}$ is a $H$-coset if there exists some $s_0 \in S$ such that $S=\{s \in \mathcal{F}_{a,b}:s-s_0 \in H\}$.

Now let $(a,b)$ be an acceptable pair such that $ab$ is not divisible by any prime congruent to $3$ modulo $4$. Write $a=t_a\prod_{l \mid a} l$ and $b=t_b\prod_{l' \mid b} l'$, where the products run over all odd primes $l\mid a$ and $l'\mid b$. 
Define $\Gamma_{\mathbb{F}_2}^{\text{unr}}(\Q,\{a,b\})$ to be the subgroup of $\Gamma_{\mathbb{F}_2}(\Q,\{a,b\})$ generated by the set $\{\chi_p:p \mid ab\} \cup \{\chi_{t_a},\chi_{t_b}\}$.

One calls an element $L \in \mathcal{F}_{a,b}$ \emph{minimally ramified} if it satisfies the following two properties. Firstly $L/\Q(\sqrt{a},\sqrt{b})$ does not ramify above any finite place $v\nmid\gcd(\Delta_{\Q(\sqrt{a})},\Delta_{\Q(\sqrt{b})})$. Secondly if one element of $\{a,b\}$ is even and the other is $5$ modulo $8$, we ask that $L/\Q(\sqrt{ab})$ is $2$-minimally ramified, see \cite[Definition~7.3]{Redei-Stevenhagen}. 

 We denote by $\mathcal{F}_{a,b}^{\text{unr}}$ the subset of $\mathcal{F}_{a,b}$ consisting of minimally ramified elements. As it is shown in \cite[Lemma~7.5]{Redei-Stevenhagen}, the set $\mathcal{F}_{a,b}^{\text{unr}}$ is a $\Gamma^{\text{unr}}(\Q,\{a,b\})$-coset (which in particular implies that it is non-empty). 
 
\begin{mydef}
Let $(a,b,c)$ be a triple with $ a,b,c \in (\Q^{*})/(\Q^{*})^2$.
We say that $(a,b,c)$ is jointly unramified if 
\[
\gcd(\Delta_{\Q(\sqrt{a})}, \Delta_{\Q(\sqrt{b})}, \Delta_{\Q(\sqrt{c})})=1.
\] 
We say that $(a,b,c)$ is admissible, if all $(a,b),(a,c),(b,c)$ are acceptable pairs, $abc$ is not divisible by any prime congruent to $3$ modulo $4$, and $(a,b,c)$ is jontly unramified. 
\end{mydef}

Observe that if a triple is admissible then so is any permutation of it. 

\begin{mydef}
For any admissible triple $(a,b,c)$, define the R\'{e}dei symbol $[a,b,c] \in \mathbb{F}_2$ as follows\footnote{We use, in contrast to \cite{Redei-Stevenhagen}, the convention that R\'{e}dei symbols take their values in $\mathbb{F}_2$, since this shall be notationally more convenient in the rest of the paper.}. 
If any of $a,b,c$ is trivial, set $[a,b,c]:=0$. 
Assuming $a,b,c$ are all non-trivial,
choose $L \in \mathcal{F}_{a,b}^{\text{unr}}$
and $\mathfrak{c}$ an integral ideal of norm $c$ in the ring of integer of $\Q(\sqrt{ab})$, existence of $\mathfrak{c}$ follows from admissibility of $(a,b,c)$.
Define 
\[[a, b, c]:=
\begin{cases}
\left[\frac{L/\Q(\sqrt{ab})}{\mathfrak{c}}\right] &\text{if }c>0 \vspace{5pt}\\
\left[\frac{L/\Q(\sqrt{ab})}{\mathfrak{c}\tilde{\infty}}\right]&\text{if }c<0.
\end{cases}
\]
We identify the Artin symbol with its image under the isomorphism $\Gal(L/\Q(\sqrt{a},\sqrt{b}))\cong \mathbb{F}_2$.
\end{mydef}

A priori the resulting symbol would depend on the choices of $L$ and $\mathfrak{c}$, and so the notation should reflect this dependency. However the following theorem shows in particular that the symbol does not depend on any of the choices. Since, in this logical structure, this independence cannot be assumed in the statement of R\'{e}dei reciprocity, the reader shall interpret every R\'{e}dei symbol appearing there as the result of one of the above choices. For a proof see \cite[Theorem~7.7]{Redei-Stevenhagen}. 
\begin{theorem} \label{thm:RR} \emph{(R\'{e}dei reciprocity)} Let $(a,b,c)$ be an admissible triple. Then
\begin{equation}\label{eq:RR}
[a,b,c]=[a,c,b].
\end{equation}
\end{theorem}

We can fix an element in $\mathcal{F}_{a,c}^{\text{unr}}$ and an integral ideal of norm $b$ in $\Q(\sqrt{ac})$, this fixes $[a,c,b]$ in \eqref{eq:RR}. We are still allowed to take any element in $\mathcal{F}_{a,b}^{\text{unr}}$ and any integral ideal of norm $c$ in $\Q(\sqrt{ab})$, and the value of $[a,b,c]$ has to be unchanged by \eqref{eq:RR}. This shows that the choices made in defining the R\'{e}dei symbol do not affect the final value of the symbol. Also observe that the symbol $[a,b,c]$ trivially does not depend on the order of the first two entries, so Theorem~\ref{thm:RR} shows that the symbol $[a,b,c]$ is invariant under any permutation of the entries. 

As a consequence of R\'{e}dei reciprocity, the following proposition shows that the R\'{e}dei symbol is linear in every entry. 

\begin{prop} 
\label{prop:bilinear}
Let $(a,b,c)$, $(a,b',c)$ be two admissible triples. Then $(a,bb',c)$ is also an admissible triple and furthermore
\[
[a,b,c]+[a,b',c]=[a,bb',c].
\]
Since admissibility and the R\'{e}dei symbol does not depend on the order of $a,b,c$ in the triple, the corresponding statements hold for all three entries.
\end{prop}

\begin{proof}
It follows from $(a,b)_v=(a,b')_v=1$ for all $v \in \Omega$ it follows and the bilinearity of Hilbert symbols, that $(a,bb')_v=1$ for all $v \in \Omega$. Therefore $(a,bb')$ is acceptable, and similarly $(bb',c)$. 
Since $(a,b,c)$ or $(a,b',c)$ are jointly ramified, we have 
\begin{multline*}
\gcd(\Delta_{\Q(\sqrt{a})}, \Delta_{\Q(\sqrt{b})}\Delta_{\Q(\sqrt{b'})}, \Delta_{\Q(\sqrt{c})})\\
=
\gcd(\Delta_{\Q(\sqrt{a})}, \Delta_{\Q(\sqrt{b})}, \Delta_{\Q(\sqrt{c})})
\gcd(\Delta_{\Q(\sqrt{a})}, \Delta_{\Q(\sqrt{b'})}, \Delta_{\Q(\sqrt{c})})
=1.
\end{multline*}
Observe that $\Delta_{\Q(\sqrt{bb'})}\mid \Delta_{\Q(\sqrt{b})}\Delta_{\Q(\sqrt{b'})}$. Therefore $(a,bb',c)$ is jointly unramified. 
It follows that $(a,bb',c)$ is an admissible triple. 

Now the desired identity follows from Theorem~\ref{thm:RR} and the linearity of the last entry. 
\end{proof}

We remark that it is possible to prove that the R\'{e}dei symbol is well-defined, and Proposition~\ref{prop:bilinear}, without using R\'{e}dei reciprocity. It is precisely this approach that works in the generality of \cite[Theorem~2.8]{Smith}. The resulting argument is substantially more involved, so for brevity we opted to use the proofs with R\'{e}dei reciprocity. Note that Theorem~\ref{t2.8self} and Theorem~\ref{t.2.8swapped} have no analogues in \cite{Smith}.

We need a final fact that will be crucial in the proof of Theorem~\ref{t2.8self}. 

\begin{prop} 
\label{prop:strick}
Let $(a,b,c)$ be an admissible triple such that $a,b>0$ and
\[
\mathrm{gcd}(\Delta_{\Q(\sqrt{a})}, \Delta_{\Q(\sqrt{b})}) = 1.
\]
Then $(a,b,-abc)$ is also admissible and 
\[
[a,b,c]=[a,b,-abc].
\]
\end{prop}

\begin{proof}
Assume that $a,b$ are both non-trivial, otherwise the statement is immediate.

We first show that $(a,b,-ab)$ is admissible. The condition of being jointly unramified follows immediately from the assumption that $\Delta_{\Q(\sqrt{a})}$ and $\Delta_{\Q(\sqrt{b})}$ are coprime. Since $a$ and $b$ are positive and are not divisible by any prime congruent to $3$ modulo $4$, it follows that $(a,-1)$ and $(b,-1)$ are both acceptable. This shows that $(a,b,-ab)$ is admissible.

We next claim that $[a,b,-ab]=0$. Let us now pick $L$ in $\mathcal{F}_{a,b}^{\text{unr}}$. Since $\gcd(\Delta_{\Q(\sqrt{a})},\Delta_{\Q(\sqrt{b})})=1$, it follows that $L/\Q(\sqrt{ab})$ is unramified at all finite places. Furthermore $L/\Q(\sqrt{ab})$ is a cyclic degree $4$ extension. On the other hand, the principal ideal $(\sqrt{ab})$ generates the kernel of the natural surjection $\text{Cl}^{+}(\Q(\sqrt{ab})) \twoheadrightarrow \text{Cl}(\Q(\sqrt{ab}))$. The extension $L/\Q(\sqrt{ab})$ is totally real if and only if this kernel acts trivially on $L$ via the Artin map. Therefore
\[
\left[\frac{L/\Q(\sqrt{ab})}{(\sqrt{ab})}\right] = \left[\frac{L/\Q(\sqrt{ab})}{\tilde{\infty}}\right].
\]
Hence $[a,b,-ab]=0$.

By Proposition~\ref{prop:bilinear}, we have that $(a,b,-abc)$ is also admissible and
\[
[a,b,c]=[a,b,c]+[a,b,-ab]=[a,b,-abc].
\qedhere
\]
\end{proof}

\subsection{Reflection principles} 
We begin by recalling the connection between R\'{e}dei symbols and $8$-rank pairings. 

Recall that $\CL^+(D)[2]$ is generated by the primes above the rational primes ramifying in $\Q(\sqrt{D})/\Q$. For each $b \mid \Delta_{\Q(\sqrt{D})}$, not necessarily positive, we define $\B_D(b)$ to be the unique integral ideal of $\mathcal{O}_{\Q(\sqrt{D})}$ having norm equal to $|b|$, if $b>0$. If $b<0$, we instead put $\B_D(b):=\B_D(|b|) \cdot (\sqrt{D})$. Recall that $\B_D(b) \in 2\CL^+(D)[4]$ if and only if $(b,D)$ forms an acceptable pair, i.e. $(b,D)_v=1$ for all $v \in \Omega$. 

Recall that $\CL^+(D)^{\vee}[2]$ is generated by $\chi_p$ with $p$ prime dividing $\Delta_{\Q(\sqrt{D})}$. Furthermore for a positive divisor $a \mid \Delta_{\Q(\sqrt{D})}$, we have that $\chi_{a} \in 2\CL^+(D)^{\vee}$ if and only if $(a,-D)$ is an acceptable pair, i.e. $(a,-D)_v=1$ for all $v \in \Omega$.

Since $D$ is not divisible by any primes congruent to $3\bmod 4$, we have for any positive $a\mid D$ we have $(a,D)_v=(a,-D)_v$. Therefore the same set of divisors of $\Delta_{\Q(\sqrt{D})}$ describes both $2\CL^+(D)[4]$ and $2\CL^+(D)^{\vee}[4]$, up to sign. In particular for any positive $a\mid D$, we have
\begin{equation}\label{eq:samespace}
\chi_{a} \in 2\CL^+(D)^{\vee} 
\quad \text{ if and only if }\quad 
\B_D(a) \in 2\CL^+(D)[4]
\end{equation}
 
Now let $a,b\mid \Delta_{\Q(\sqrt{D})}$ such that $\chi_a \in 2\CL^+(D)^{\vee}[4]$ and $\B_D(b) \in 2\CL^+(D)[4]$. Then for all cyclic degree $4$ extensions $L/\Q(\sqrt{D})$ unramified at all finite places and containing $\Q(\sqrt{a},\sqrt{D})$, the Artin symbol $[\frac{L/\Q(\sqrt{D})}{\B_D(b)}]$ always lands in the unique cyclic subgroup of order $2$ of $\Gal(L/\Q(\sqrt{D}))$, since $\B_D(b)\in\CL^+(D)[2]$. Furthermore, for a fixed $a$, the value of the symbol does not depend on the choice of $L$, since $\B_D(b)\in 2\CL^+(D)[4]$. In this statement we are implicitly identifying in the unique possible way any two groups of size $2$. The value of this symbol is by definition
\[
\langle \chi_a, b \rangle_{D}
\]
and we shall refer to it as the \emph{Artin pairing} between $\chi_a$ and $b$. 

We next define spaces $\mathcal{C}(D)$ and $\mathcal{C}^{\vee}(D)$ as follows. We define $\mathcal{C}^{\vee}(D)$ to be the subgroup of $\Q^{*}/(\{1,D\}(\Q^{*})^2)$ given by
\[
\mathcal{C}^{\vee}(D):=\left\{a \mid \Delta_{\Q(\sqrt{D})}: a>0\right\},
\]
and we define $\mathcal{C}(D)$ to be the subgroup of $\Q^{*}/(\{1,-D\}(\Q^{*})^2)$ given by
\[
\mathcal{C}(D):=\left\{b \mid \Delta_{\Q(\sqrt{D})}\right\}.
\] 
We put $2\mathcal{C}^{\vee}(D)$ the preimage of $2\CL^+(D)$ in $\mathcal{C}^{\vee}(D)$ and similarly for $2\mathcal{C}(D)$. Hence we have defined a pairing $2\mathcal{C}^{\vee}(D) \times 2\mathcal{C}(D) \to \mathbb{F}_2$ via the assignment $(a,b) \mapsto \langle \chi_a, b \rangle_{D}$. The fundamental property of this pairing, which can be verified easily, is that the left (resp.\ the right) kernel of this pairing is the preimage of $4\CL^+(D)^{\vee}$ (resp. of $4\CL^+(D)$) in $\mathcal{C}^{\vee}(D)$ (resp.\ in $\mathcal{C}(D)$). Another crucial feature of the pairing is that it can be computed using R\'{e}dei symbols. 
 
\begin{prop} 
\label{prop:Redei8rk}
Let $(a,b)$ be a pair with $a,b \in \Q^{*}/(\Q^{*})^2$ and such that $\Delta_{\Q(\sqrt{a})},\Delta_{\Q(\sqrt{b})}$ are coprime. Furthermore assume that $a,b>0$ and not divisible by any prime congruent to $3$ modulo $4$. Let $c$ be a (not necessarily positive) divisor of $\Delta_{\Q(\sqrt{ab})}$. Assume that $\chi_{a} \in 2\CL^+(ab)^{\vee}[4]$ and  $\B_{ab}(c) \in 2\CL^+(ab)[4]$. Then the triple $(a,b,c)$ is admissible and we have that
\[
\langle \chi_a,c \rangle_{ab}=[a,b,c].
\]
\end{prop}

\begin{proof}
Observe that $(a,b)$ and $(ab,c)$ are acceptable since $\chi_{a} \in 2\CL^+(ab)^{\vee}[4]$ and $\B_{ab}(c) \in 2\CL^+(ab)[4]$. 

We claim that $(a,c)$ is acceptable. A similar argument shows that $(b,c)$ is acceptable. Firstly $a>0$ implies $(a,c)_{\infty}=1$.
Now we check that $(a,c)_{v}=1$ for all $v \in \Omega$ finite and odd.
If $v\nmid ac$, trivially we have $(a,c)_v=1$.
If $v$ divides only $a$ but not $c$, we have that $(a,c)_v=(ab,c)_v=1$. 
If $v$ divides only $c$ but not $a$, we have that $(a,c)_v=(a,ab)_v=1$. 
Now assume that $v$ divides both $a$ and $c$.
Since $\Delta_{\Q(\sqrt{a})},\Delta_{\Q(\sqrt{b})}$ are coprime, we must have $v\nmid b$. Also by assumption $v^2$ cannot divide $a$ or $c$, so $(b,ac)_v=1$.
Therefore 
$(a,c)_v=(a,ac)_v=(ab,ac)_v$.
Since $(a,b)$ and $(ab,c)$ are acceptable, we have $(a,ab)_v=(a,b)_v=(ab,c)_v=1$, so $(ab,ac)_v=1$, as required.
The remaining case $v=2$ follows from Hilbert reciprocity. This shows that $(a,c)$ and similarly $(b,c)$ are acceptable pairs. 

Since $a,b$ are coprime, and not divisible by any prime congruent to $3\bmod 4$, we conclude that $\gcd(\Delta_{\Q(\sqrt{a})}, \Delta_{\Q(\sqrt{b})}, \Delta_{\Q(\sqrt{c})})=1$. Therefore the triple $(a,b,c)$ is admissible. 

Now observe that any $L \in \mathcal{F}_{a,b}^{\text{unr}}$ gives a cyclic degree $4$ extension of $\Q(\sqrt{ab})$ that is unramified at all finite places and contains $\Q(\sqrt{a},\sqrt{b})$. Therefore $\langle \chi_a, c \rangle_{ab}=[\frac{L/\Q(\sqrt{ab})}{\B_{ab}(c)}]=[a,b,c]$.
\end{proof}

We are now ready to prove our main algebraic results. 

\begin{theorem}
\label{t2.8}
Let $d$ be a positive squarefree integer composed of primes that are $1$ or $2$ modulo $4$. Let $p_1, p_2, q_1, q_2$ be primes that are $1$ modulo 4 and coprime to $d$. Let $a$ be a positive divisor of $d$, and let $b$ be any (possibly negative) divisor of $d$. Assume that
\[
\B_{p_iq_jd}(b) \in 2\CL^+(p_iq_jd)[4] \text{ for all } i, j \in \{1, 2\}.
\]
\begin{enumerate}
\item Suppose
\[
\chi_a \in 2\CL^+(p_iq_jd)^\vee[4] \text{ for all } (i, j) \in \{(1, 2), (2, 1), (2, 2)\}
.\]
Then we have $\chi_a \in 2\CL^+(p_1q_1d)^\vee[4]$ and 
\begin{equation}\label{eq:t2.8i}
\langle \chi_a, b \rangle_{p_1q_1d} + \langle \chi_a, b \rangle_{p_1q_2d} + \langle \chi_a, b \rangle_{p_2q_1d} + \langle \chi_a, b \rangle_{p_2q_2d} = 0.
\end{equation}
\item Suppose instead 
\[
\chi_{p_ia} \in 2\CL^+(p_iq_jd)^\vee[4] \text{ for all } (i, j) \in \{(1, 2), (2, 1), (2, 2)\},
\]
and 
\[
\left(\frac{q_1q_2}{p_1}\right)=\left(\frac{p_1p_2}{q_1}\right)=1.
\]
Then $\chi_{p_1a} \in 2\CL^+(p_1q_1d)^\vee[4]$. Furthermore the triple $(p_1p_2,q_1q_2,b)$ is admissible and
\begin{equation}\label{eq:t2.8ii}
\langle \chi_{p_1a}, b \rangle_{p_1q_1d} + \langle \chi_{p_1a}, b \rangle_{p_1q_2d} + \langle \chi_{p_2a}, b \rangle_{p_2q_1d} + \langle \chi_{p_2a}, b \rangle_{p_2q_2d} = [p_1p_2, q_1q_2, b].
\end{equation}
\end{enumerate}
\end{theorem}

\begin{proof}
$(i)$ Recall that $\chi_{a} \in 2\CL^+(p_1q_1d)^{\vee}[4]$ is equivalent to $(a,p_1q_1\frac{d}{a})$ being an acceptable pair. For each $v \in \Omega$ we already know that $(a,p_1q_2\frac{d}{a})_v=(a,p_2q_1\frac{d}{a})_v=(a,p_2q_2\frac{d}{a})_v=1$. Therefore taking their product, we obtain by the bilinearity of Hilbert symbols that $(a,p_1q_1\frac{d}{a})_v=1$ for each $v \in \Omega$, so $\chi_{a} \in 2\CL^+(p_1q_1d)^{\vee}[4]$ as desired. 

Now by Proposition~\ref{prop:Redei8rk}, we obtain that the four triples $(a,p_1q_1 \frac{d}{a},b)$, $(a,p_1q_2\frac{d}{a},b)$, $(a,p_2q_1\frac{d}{a},b)$ and $(a,p_2q_2\frac{d}{a},b)$ are all admissible, and the left-hand side of \eqref{eq:t2.8i} equals
\[[a,p_1q_1 \frac{d}{a},b]+[a,p_1q_2\frac{d}{a},b]+[a,p_2q_1\frac{d}{a},b]+[a,p_2q_2\frac{d}{a},b].
\]
By Proposition~\ref{prop:bilinear}, this sum equals
\[[a,q_1q_2,b]+[a,q_1q_2,b]=0.
\]

$(ii)$ Recall that $\chi_{p_1a} \in 2\CL^+(p_1q_1d)^{\vee}[4]$ is equivalent to $(p_1a,q_1\frac{d}{a})$ being an acceptable pair. 
For each $v \in \Omega$ we already know that $(p_1a,q_2\frac{d}{a})_v=(p_2a,q_1\frac{d}{a})_v=(p_2a,q_2\frac{d}{a})_v=1$. Therefore taking their product, we obtain by the bilinearity of Hilbert symbols that $(p_1a,q_1\frac{d}{a})_v=(p_1p_2,q_1q_2)_v$ for each $v \in \Omega$.
The symbol $(p_1p_2,q_1q_2)_v$ is trivial at all odd primes $v\nmid p_1p_2q_1q_2$, and also at $2$ and $\infty$ since $p_1,p_2,q_1,q_2$ are positive and congruent to $1$ modulo $4$.
Also note that $(p_1a,q_1\frac{d}{a})_v=1$ for $v= p_2, q_2$, since $p_2,q_2\nmid p_1q_1d$.
Therefore it remains to check $v=p_1$ and $q_1$.
At $v=p_1$, this becomes $(p_1,q_1q_2)_{p_1}=\leg{q_1q_2}{p_1}$ which is trivial by assumption. The case $v=q_1$ is similar.
This shows that $(p_1a,q_1\frac{d}{a})_v=(p_1p_2,q_1q_2)_v=1$ for each $v \in \Omega$, and $\chi_{p_1a} \in 2\CL^+(p_1q_1d)^{\vee}[4]$.

Next by Proposition~\ref{prop:Redei8rk} we know that the triples $(p_1a,q_1 \frac{d}{a},b),(p_1a,q_2\frac{d}{a},b),(p_2a,q_1\frac{d}{a},b)$ and $(p_2a,q_2\frac{d}{a},b)$ are all admissible, and that the left-hand side of \eqref{eq:t2.8ii} equals
\[
[p_1a,q_1 \frac{d}{a},b]+[p_1a,q_2\frac{d}{a},b]+[p_2a,q_1\frac{d}{a},b]+[p_2a,q_2\frac{d}{a},b].
\]
Applying Proposition~\ref{prop:bilinear}, we find that $(p_1a,q_1q_2,b)$ and $(p_2a,q_1q_2,b)$ are also admissible and this sum equals
\[
[p_1a,q_1q_2,b]+[p_2a,q_1q_2,b].
\]
Another application of Proposition~\ref{prop:bilinear}, shows that $(p_1p_2,q_1q_2,b)$ is admissible and the above sum is 
\[[p_1p_2,q_1q_2,b].
\qedhere
\]
\end{proof}

\begin{theorem}
\label{t2.8self}
Let $d$ be a positive squarefree integer composed of primes that are $1$ or $2$ modulo $4$. Take primes $p_1, p_2, q_1, q_2$ that are $1$ modulo 4 and coprime to $d$. Let $a$ be a positive divisor of $d$. We assume that
\[
\B_{p_iq_jd}(p_ia) \in 2\CL^+(p_iq_jd)[4] \text{ for all } i, j \in \{1, 2\}.
\]
Then we have 
\[
\chi_{p_ia} \in 2\CL^+(p_iq_jd)^\vee[4]  \text{ for all } i, j \in \{1, 2\}.
\]
Moreover, the triple $(p_1p_2,q_1q_2,p_1p_2)$ is admissible and
\begin{equation}\label{eq:t2.8self}
\langle \chi_{p_1a}, p_1a \rangle_{p_1q_1d} + \langle \chi_{p_1a}, p_1a \rangle_{p_1q_2d} + \langle \chi_{p_2a}, p_2a \rangle_{p_2q_1d} + \langle \chi_{p_2a}, p_2a \rangle_{p_2q_2d} = [p_1p_2, q_1q_2, p_1p_2].
\end{equation}
\end{theorem}

\begin{proof}
By \eqref{eq:samespace}, the assumption $\B_{p_iq_jd}(p_ia) \in 2\CL^+(p_iq_jd)[4]$ implies that $\chi_{p_ia} \in 2\CL^+(p_iq_jd)^{\vee}[4]$ for each $i,j \in \{1,2\}$.

By Proposition~\ref{prop:Redei8rk}, we conclude that $(p_1a,q_1\frac{d}{a},p_1a)$, 
$(p_1a,q_2\frac{d}{a},p_1a)$,  
$(p_2a,q_1\frac{d}{a},p_2a)$ and 
$(p_2a,q_2\frac{d}{a},p_2a)$ are all admissible, and furthermore the left-hand side of \eqref{eq:t2.8self} is
\[
[p_1a,q_1\frac{d}{a},p_1a]+[p_1a,q_2\frac{d}{a},p_1a]+[p_2a,q_1\frac{d}{a},p_2a]+[p_2a,q_2\frac{d}{a},p_2a].
\]
Using Proposition~\ref{prop:bilinear}, we have that $(p_1a,q_1q_2,p_1a)$ and $(p_2a,q_1q_2,p_2a)$ are admissible triples and the sum becomes
\[
[p_1a,q_1q_2,p_1a]+[p_2a,q_1q_2,p_2a].
\]
Next, since $p_1a,q_1q_2$ are coprime and  $p_2a,q_1q_2$ are coprime, by Proposition~\ref{prop:strick}, $(p_1a,q_1q_2,-q_1q_2)$ and $(p_2a,q_1q_2,-q_1q_2)$ are admissible and the above sum is
\[
[p_1a,q_1q_2,-q_1q_2]+[p_2a,q_1q_2,-q_1q_2].
\]
By Proposition~\ref{prop:bilinear}, $(p_1p_2,q_1q_2,-q_1q_2)$ is admissible and
the above sum is 
\[[p_1p_2,q_1q_2,-q_1q_2]\]
Since $p_1p_2,q_1q_2$ are coprime, applying Proposition~\ref{prop:strick} again shows that $(p_1p_2,q_1q_2,p_1p_2)$ is admissible and
\[[p_1p_2,q_1q_2,-q_1q_2]=
[p_1p_2,q_1q_2,p_1p_2],
\]
which gives the desired result. 
\end{proof}

\begin{theorem} 
\label{t.2.8swapmin}
Let $d$ be a positive squarefree integer composed of primes that are $1$ or $2$ modulo $4$. Let $p_1, p_2, q_1, q_2$ be distinct primes that are $1$ modulo 4 and coprime to $d$. Let $a,b$ be a positive divisors of $d$. We assume that
\[
\B_{p_iq_jd}(b), \B_{p_iq_jd}(p_ia) \in 2\CL^+(p_iq_jd)[4] \text{ for all } i, j \in \{1, 2\}.
\]
Then we have that
\[
\chi_{b},\chi_{p_ia} \in 2\CL^+(p_iq_jd)^{\vee}[4] \text{ for all } i, j \in \{1, 2\}.
\]
Furthermore we have that
\[
\sum_{i,j \in \{1,2\}}\langle \chi_{p_ia}, b \rangle_{p_iq_jd}+\langle \chi_{b},p_ia \rangle_{p_iq_jd}=0.
\]
\end{theorem}

\begin{proof}
By \eqref{eq:samespace}, the assumption $\B_{p_iq_jd}(b), \B_{p_iq_jd}(p_ia) \in 2\CL^+(p_iq_jd)[4]$ implies that $\chi_{b},\chi_{p_ia} \in 2\CL^+(p_iq_jd)^{\vee}[4]$ for all $i, j \in \{1, 2\}$. 

By Proposition~\ref{prop:Redei8rk}, the triples $(p_ia,\frac{d}{a}q_j,b)$ and $(b,\frac{d}{b}p_iq_j,ap_i)$ are admissible for all choices of $i,j$ in $\{1,2\}$. Furthermore the sum of the pairings in this proposition can be rewritten as
\[
\sum_{i,j\in\{1,2\}}[p_ia,\frac{d}{a}q_j,b]
+[b,\frac{d}{b}p_iq_j,ap_i]
\]
Applying Proposition~\ref{prop:bilinear} we can rewrite this as
\[[p_1a,q_1q_2,b]+[p_2a,q_1q_2,b]+[b,q_1q_2,ap_1]+[b,q_1q_2,ap_2]=[p_1p_2,q_1q_2,b]+[b,q_1q_2,p_1p_2]=0.
\]
The first equality follows from Proposition~\ref{prop:bilinear} and the last equality follows from applying Theorem~\ref{thm:RR}. 
\end{proof}

\begin{theorem} 
\label{t.2.8swapped}
Let $d$ be a positive squarefree integer composed of primes that are $1$ or $2$ modulo $4$. Let $p_1, p_2, q_1, q_2$ be distinct primes that are $1$ modulo 4 and coprime to $d$. Let $a,b$ be positive divisors of $d$. We assume that
\[
\B_{p_iq_jd}(q_jb), \B_{p_iq_jd}(p_ia) \in 2\CL^+(p_iq_jd)[4] \text{ for all } i, j \in \{1, 2\}.
\]
Then we have that
\[
\chi_{q_jb},\chi_{p_ia} \in 2\CL^+(p_iq_jd)^{\vee}[4] \text{ for all } i, j \in \{1, 2\}.
\]
Furthermore the triple $(p_1p_2,q_1q_2,-1)$ is admissible and
\begin{equation}\label{eq:t.2.8swapped}
\sum_{i, j \in \{1,2\}} \langle \chi_{p_ia}, q_jb \rangle_{p_iq_jd} + \langle \chi_{q_jb},p_ia \rangle_{p_iq_jd}= [p_1p_2, q_1q_2, -1].
\end{equation}
\end{theorem}
\begin{proof}
By \eqref{eq:samespace}, the assumption $\B_{p_iq_jd}(q_jb), \B_{p_iq_jd}(p_ia) \in 2\CL^+(p_iq_jd)[4]\in 2\CL^+(p_iq_jd)[4]$ implies that $\chi_{p_ia},\chi_{q_jb} \in 2\CL^+(p_iq_jd)^{\vee}[4]$ for all $i, j \in \{1, 2\}$. 

By Proposition~\ref{prop:Redei8rk}, we have that the triples $(p_ia,\frac{d}{a}q_j,q_jb),(q_jb,\frac{d}{b}p_i,p_ia)$ are admissible for each choice of $i,j$ in $\{1,2\}$ and left-hand side of \eqref{eq:t.2.8swapped} equals
\[
\sum_{i,j \in \{1,2\}}[p_ia,\frac{d}{a}q_j,q_jb]+[q_jb,\frac{d}{b}p_i,p_ia].
\]
By Proposition~\ref{prop:bilinear} we can rewrite this sum of R\'{e}dei symbols as
\[
[p_1p_2,\frac{d}{a}q_1,bq_1]+[p_1p_2,\frac{d}{a}q_2,bq_2]+[q_1q_2,\frac{d}{b}p_1,ap_1]+[q_1q_2,\frac{d}{b}p_2,ap_2].
\]
One readily checks that $p_i\frac{d}{b}$ is coprime to $q_1q_2$ and that $q_j\frac{d}{a}$ is coprime to $p_1p_2$. Therefore we can apply Proposition~\ref{prop:strick} to each of the terms in the above sum
\[[p_1p_2,\frac{d}{a}q_1,-dabp_1p_2]+[p_1p_2,\frac{d}{a}q_2,-dabp_1p_2]+[q_1q_2,\frac{d}{b}p_1,-dabq_1q_2]+[q_1q_2,\frac{d}{b}p_2,-dabq_1q_2].
\]
Applying Proposition~\ref{prop:bilinear} we can further simplify this and get
\[
[p_1p_2,q_1q_2,-dabp_1p_2]+[p_1p_2,q_1q_2,-dabq_1q_2]=[p_1p_2,q_1q_2,p_1p_2q_1q_2].
\]
Since $p_1p_2$ and $q_1q_2$ are coprime, we can apply Proposition~\ref{prop:strick} and get that $(p_1p_2,q_1q_2,-1)$ is admissible and the above R\'{e}dei symbol equals 
\[
[p_1p_2,q_1q_2,-1],
\]
as required. 
\end{proof}

\section{A combinatorial result}
Let $X_1, \ldots, X_m$ be finite, non-empty sets and let $X := X_1 \times \ldots \times X_m$. Put
\[
V := \{F: X \rightarrow \FF_2\}, \quad W := \{g: X \times X \rightarrow \FF_2\}.
\]
Given two elements $x_1, x_2 \in X$ and $\mathbf{v} \in \{1, 2\}^m$, we define $\mathbf{v}(x_1, x_2)$ to be the unique element $y \in X$ such that $\pi_j(y) = \pi_j(x_{\pi_j(\mathbf{v})})$. We also define a linear map $d: V \rightarrow W$ given by
\[
dF(x_1, x_2) = \sum_{\mathbf{v} \in \{1, 2\}^m} F(\mathbf{v}(x_1, x_2)).
\]
We define $\mathcal{A}(X) := \im(d)$.

\begin{lemma}
\label{lIm}
We have that
\[
\dim_{\FF_2} \mathcal{A}(X) = \prod_{i = 1}^m \left(|X_i| - 1\right).
\]
\end{lemma}

\begin{proof}
See Proposition~9.3 in the work of Koymans and Pagano \cite{KP}.
\end{proof}

\begin{mydef}
Let $\epsilon > 0$ be given. We say that $F$ is $\epsilon$-bad if
\[
\left|F^{-1}(0) - \frac{|X|}{2}\right| \geq \epsilon |X|.
\]
We say that $g \in \mathcal{A}(X)$ is $\epsilon$-bad if there is $\epsilon$-bad $F$ such that $dF = g$.
\end{mydef}

In our application we shall be able to prove distributional properties of $g$ by using the Chebotarev Density Theorem. However, we have no direct control over $F$ itself. Nevertheless, the following theorem will allow us to prove the desired equidistribution for $F$. Note the similarity to Proposition~4.3 in Smith \cite{Smith}. Since we are dealing with the $8$-rank, we shall not need the more complicated Proposition~4.4 in Smith \cite{Smith}, and this allows us to save two logarithms.

\begin{theorem}
\label{tBadg}
Let $\epsilon > 0$ be given. Then we have
\[
\frac{|\{g \in \mathcal{A}(X) : g \textup{ is } \epsilon\textup{-bad}\}|}{|\mathcal{A}(X)|} \leq 2^{1 + |X| - \prod_{i = 1}^m \left(|X_i| - 1\right)} \cdot \exp(-2\epsilon^2|X|).
\]
\end{theorem}

\begin{proof}
Hoeffding's inequality shows that the proportion of $F$ that are $\epsilon$-bad is at most
\begin{align}
\label{eHoeffF}
\frac{|\{F \in V : F \textup{ is } \epsilon\textup{-bad}\}|}{|V|} \leq 2\exp(-2\epsilon^2|X|).
\end{align}
Define
\[
a := |X| - \prod_{i = 1}^m \left(|X_i| - 1\right).
\]
By Lemma~\ref{lIm} we see that the kernel of $d$ is an $a$-dimensional vector space. Combining this with equation~\eqref{eHoeffF}, we infer that
\[
\frac{|\{g \in \mathcal{A}(X) : g \textup{ is } \epsilon\textup{-bad}\}|}{|\mathcal{A}(X)|} \leq \frac{|\{F \in V : F \textup{ is } \epsilon\textup{-bad}\}|}{|\mathcal{A}(X)|} \leq 2^{a + 1} \cdot \exp(-2\epsilon^2|X|),
\]
which is the theorem.
\end{proof}

\section{Prime divisors}
\label{sPrimes}
In \cite[Section 5]{Smith}, Smith proved that several properties pertaining to the spacing of prime divisors of integers in the set
$\{1\leq n\leq N:\ \omega(n) = r,\ p \mid n\Rightarrow p>D\}$
occur frequently. 
We will obtain similar results on squarefree integers with no prime factor congruent to $3\bmod 4$.

Define
$S(x):=\{n<x: p\mid n \Rightarrow p\not\equiv 3\bmod 4,\ n\text{ squarefree}\}$, 
$S_r(x):=\{n\in S(x): \omega(n)=r\}$ and $\mu=\frac{1}{2}\log\log x$.
A classical result by Landau \cite{Landau} shows that
\[\Phi(x):=\#S(x)\asymp \frac{x}{\sqrt{\log x}}.\]
Noting the prime number theorem for arithmetic progressions
\[\#\{p\leq x: p\equiv 1\bmod 4\}=\frac{1}{2}\mathrm{li}(x)+O\left(x\exp\left(-c\sqrt{\log x}\right)\right),\]
one can deduce as in Sath\'{e}--Selberg theorem, uniformly in the range $r<2\mu$, we have
\[\Phi_r(x):=\#S_r(x)\asymp \frac{x}{\log x}\frac{\left(\frac{1}{2}\log\log x\right)^{r-1}}{(r-1)!}.\]
This shows that the number of distinct prime factors is Poisson distributed in $S(x)$. 

By Erd\H{o}s--Kac theorem \cite[Proposition~3]{SieveErdosKac}, the density of integers in  $S(x)$ with $|r-\mu|>\mu^{2/3}$, is 
$\ll \mu^{-1/6}\exp\left(-\frac{1}{2}\mu^{1/3}\right)\ll \exp\left(-\frac{1}{2}\mu^{1/3}\right)$.
We will make use of the following bound on the tail of the standard normal distribution 
\[\mathrm{Prob}(\mathrm{Normal}(0,1)>z)\leq \frac{\exp(-z^2/2)}{z\sqrt{2\pi}}.\]

In the following, for any $n\in S(x)$, write $r=\omega(n)$ and list the distinct prime factors of $n$ as $p_1<p_2<\dots<p_r$.

We will prove that almost all $n\in S_r(x)$ has three particular types of spacing.
\begin{theorem}\label{theorem:small}
Let $\epsilon>0$.
Take $y_1>3$ and $\eta >1$. Assume 
\begin{equation}\label{eq:rrange}
|r-\mu|<\mu^{2/3}.
\end{equation}
 Then
\begin{enumerate}
\item other than $\ll\Phi_r(x)\left((\log y_1)^{-1}+(\log x)^{-1/2+\epsilon}\right)$
exceptions, all
$n\in S_r(x)$ are comfortably spaced above $y_1$: $2y_1<p_i<p_{i+1}/2$ for any $p_i>y_1$;
\label{part1}
\item other than
$\ll\Phi_r(x)\exp(-k\eta )$
exceptions, where $k$ is an absolute constant,
all
$n\in S_r(x)$ are $\eta $-regularly spaced:
\[\left|\frac{1}{2}\log\log p_i-i\right|<\eta ^{1/5}\max\{i,\eta \}^{4/5}\text{ for all } i<\frac{1}{3}r;\]\label{part2}
\item 
other than 
$\ll \Phi_r(x)\exp\left(-(\log \log\log x)^{1/3-\epsilon}\right)$
exceptions, 
all
$n\in S_r(x)$ are extravagantly spaced:
\[\log p_i\geq (\log\log p_i)^2\cdot \log\log \log x\cdot \sum_{j=1}^{i-1}\log p_j\text{ for some } \frac{1}{2}r^{1/2}<i<\frac{1}{2}r.\]\label{part3}
\end{enumerate} 
\end{theorem}

\subsection{Some estimates}
\subsubsection{Upper bound for rough numbers}
Mertens' theorem shows that there exists constants $c,M>0$ such that for any $x>2$,
\[\sum_{\substack{p\leq x\\p\not\equiv 3\bmod 4}}\frac{1}{p}
=\frac{1}{2}\log\log x+M+O\left(x\exp\left(-c\sqrt{\log x}\right)\right).
\]
Fixing some large enough absolute constant $B_1>0$, we have for any $x>2$
\[\frac{1}{2}\log\log x- B_1\leq \sum_{\substack{p\leq x\\p\not\equiv 3\bmod 4}}\frac{1}{p}
\leq \frac{1}{2}\log\log x+ B_1.\]

For any set of primes $E$, define 
\[E(x):=\sum_{\substack{p\leq x\\ p\in E}}\frac{1}{p}.\]
We will apply the following theorem by Tudesq \cite[Theorem~2]{Tudesq}.
\begin{theorem}\label{theorem:Tudesq}
There exists absolute constant $B_2>0$ such that 
\[\#\{n\leq x: \omega_{E_j}(n)=k_j,\ 0\leq j\leq i\}
\ll x\exp\left(-\sum_{j=0}^l E_j(x)\right) \prod_{j=0}^l\frac{\left(E_j(x)+B_2\right)^{k_j}}{k_j!}
\]
for all $x\geq 1$, $l\geq 0$, $E_j$ pairwise disjoint sets of primes, $k_j\geq 0$.
\end{theorem}

In our application we will take $E_0$ to be the set of primes congruent to $3\bmod 4$ and $E_0,E_1,\cdots,E_l$ to be pairwise disjoint sets of primes so that $\cup_{j=0}^l E_j$ contains all primes. Also take $k_0=0$, and $k_1+\cdots +k_l=r$.
Then 
\[\#\{n\in S_r(x): \omega_{E_j}(n)=k_j,\ 1\leq j\leq l\}
\ll \frac{x}{\log x}\prod_{j=1}^l\frac{\left(E_j(x)+B_2\right)^{k_j}}{k_j!}.
\]

\subsubsection{Upper bound for smooth numbers}
Define 
\[\Psi_r(x,y):=\#\{n\in S_r(x):\ p\mid n\Rightarrow p<y\},\]
which is the size of the set of $y$-smooth numbers in $S_r(x)$.

We will need an upper bound for smooth numbers for small $u:=\log x/\log y$. There are works the number of prime factors of smooth numbers \cite{Alladi,Hensley,Hildebrand}, but none of which explicitly gives a formula for the range of small $u$ we are interested in. We prove an upper bound here that is sufficient for our application, although more work could be done to obtain a more precise estimate.

\begin{lemma}\label{lemma:ubsmooth}
Fix some $\epsilon>0$. There exists some large enough $A>0$ such that the following holds.
Take $x>y>2$ and some integer $k\geq 1$ such that $\frac{1}{2}k<\frac{1}{2}\log\log y<2k$ and $u:=\frac{\log x}{\log y}<(\log x)^{1-\epsilon}$ and assume $u>A$.
Then 
\[\Psi_k(x,y)\leq\frac{u^{-u}x}{\log y}\cdot\frac{(\frac{1}{2}\log\log y)^{k-1}}{(k-1)!}.\]
\end{lemma}

\begin{proof}
We have 
\begin{equation}\label{eq:smooth}
(\log x )\Psi_k(x,y)=\sum_{n\in\Psi_k(x,y)}\log n+\sum_{n\in\Psi_k(x,y)}\log \frac{x}{n}.
\end{equation}
We first treat the first term, which is the main contribution. Factoring each $n\in\Psi_k(x,y)$ gives
\begin{equation}\label{eq:logn}
\sum_{n\in\Psi_k(x,y)}\log n
\leq
\sum_{m\in \Psi_{k-1}(x,y)}\sum_{\substack{p<\min\left\{\frac{x}{m},y\right\}\\ p\not\equiv 3\bmod 4}} \log p
.
\end{equation}
Now taking any $0<\sigma<1$, we have
\[
\sum_{\substack{p<\min\left\{\frac{x}{m},y\right\}\\ p\not\equiv 3\bmod 4}} \log p\ll \min\left\{\frac{x}{m},y\right\}\leq 
\left( \frac{x}{m}\right)^{\sigma}y^{1-\sigma}
\]
Then writing $\frac{1}{m}=\prod_{p\mid m}\frac{1}{p}$, \eqref{eq:logn} becomes 
\[\sum_{n\in\Psi_k(x,y)}\log n\ll x^{\sigma}y^{1-\sigma}\sum_{m\in \Psi_{k-1}\left(\frac{x}{p},y\right)} \frac{1}{m^{\sigma}}\ll
\frac{x^{\sigma}y^{1-\sigma}}{(k-1)!}\left(\sum_{\substack{p<y\\ p\not\equiv 3\bmod 4}}\frac{1}{p^{\sigma}}\right)^{k-1}
.\]
Take $\sigma=1-\frac{\log (u\log u)}{\log y}$, which is
positive and tends to $1$ since $u<(\log x)^{1-\epsilon}$.
Then 
$x^{\sigma}=\frac{x}{(u\log u)^u}$
and
$y^{1-\sigma}=u\log u$.
Noting that $\li(t)=\frac{t}{\log t}+O(\frac{t}{(\log t)^2})$ and $\Ei(1/t)=\log t+O(1)$, we have
\[
\int_{e<t<y}\frac{dt}{t^{\sigma}\log t}
=\li(u\log u)-\Ei\left(\frac{\log(u\log u)}{\log y}\right)
=\log\log y+u\left(1+O\left(\frac{\log\log u}{\log u}\right)\right)
.
\]
Therefore evaluating the Stieltjes integral $\int_{t<y}\frac{d\pi(t)}{2t^{\sigma}}$ gives
\[\sum_{\substack{p<y\\ p\not\equiv 3\bmod 4}}\frac{1}{p^{\sigma}}=\frac{1}{2}\log\log y+\frac{1}{2}u\left(1+O\left(\frac{\log\log u}{\log u}\right)\right).\]
Putting together we get 
\[\sum_{n\in\Psi_k(x,y)}\log n
\ll
\frac{x}{(u\log u)^{u-1}}\cdot\frac{\left(\frac{1}{2}\log\log y+\frac{1}{2}u\left(1+O\left(\frac{\log\log u}{\log u}\right)\right)\right)^{k-1}}{(k-1)!}
.\]
The second sum in \eqref{eq:smooth} is
\[\leq \frac{x^\sigma}{\sigma}\sum_{n\in \Psi(x,y)}\frac{1}{n^\sigma}
\leq \frac{x^\sigma}{\sigma}\left(\sum_{\substack{p<y\\ p\not\equiv 3\bmod 4}}\frac{1}{p^{\sigma}}\right)^k
\ll
\frac{x}{(u\log u)^u}\cdot \frac{\left(\frac{1}{2}\log\log y+\frac{1}{2}u\left(1+O\left(\frac{\log\log u}{\log u}\right)\right)\right)^k}{k!}.
\]
Since $\log\log y/2k$ is bounded, putting back in \eqref{eq:smooth}, 
\begin{multline*}
\Psi_k(x,y)\ll
\frac{1}{(u\log u)^{u-1}}\cdot \frac{x}{\log x}\cdot \frac{\left(\frac{1}{2}\log\log y+\frac{1}{2}u\left(1+O\left(\frac{\log\log u}{\log u}\right)\right)\right)^{k-1}}{(k-1)!}\\
\ll\frac{e^{u}}{(u\log u)^{u}}\cdot\frac{x}{\log y}\cdot\frac{\left(\frac{1}{2}\log\log y\right)^{k-1}}{(k-1)!},
\end{multline*}
which implies the required result.
\end{proof}

\subsection{Proof of Theorem~\ref{theorem:small}}
\subsubsection{Proof of Theorem~\ref{theorem:small}\ref{part1}}
The number of  $n\in S_r(x)$ for which
\[y_1<p<2y_1\text{ for some } p\mid n\text{ or }y_1<q<p<2q\text{ for some } pq\mid n\]
is bounded by
\[
\sum_{\substack{y_1<p_i<2y_1\\p\equiv 1\bmod 4}}\Phi_{r-1}\left(\frac{x}{p}\right)
+\sum_{\substack{y_1<q<\sqrt{x}\\q\equiv 1\bmod 4}}\sum_{\substack{q<p<2q\\p\equiv 1\bmod 4}}\Phi_{r-2}\left(\frac{x}{pq}\right)
.\]

Split the sum into the cases $p<x^{1/4}$ and $p>x^{1/4}$.
First bound the sum $p<x^{1/4}$, assume $y_1<x^{1/4}$ otherwise the sum is zero. The sum is bounded by
\[\ll \Phi_r(x)\sum_{\substack{y_1<p<2y_1\\p\equiv 1\bmod 4}}\frac{1}{p}
+\Phi_r(x)\sum_{\substack{y_1<q<\sqrt{x}\\q\equiv 1\bmod 4}}\sum_{\substack{q<p<2q\\p\equiv 1\bmod 4}}\frac{1}{pq}\\
\ll\frac{\Phi_r(x)}{\log y_1}.\]
The sum $p>x^{1/4}$ is bounded by
\[
x\sum_{\substack{y_1<p<2y_1\\p>x^{1/4}\\p\equiv 1\bmod 4}}\frac{1}{p}+x\sum_{\substack{y_1<q<\sqrt{x}\\q\equiv 1\bmod 4}}\sum_{\substack{q<p<2q\\p>x^{1/4}\\p\equiv 1\bmod 4}}\frac{1}{pq}
\ll \frac{x}{\log x}\ll 
\frac{\Phi_r(x)}{\sqrt{\log{x}}}\sqrt{\mu}\exp\left(\frac{1}{2}\mu^{1/3}\right)
\ll 
\frac{\Phi_r(x)}{(\log{x})^{1/2-\epsilon}}
.
\]

\subsubsection{Proof of Theorem~\ref{theorem:small}\ref{part2}}
In the following let $B:=2B_1+B_2$, where $B_1,B_2$ are as defined in Theorem~\ref{theorem:Tudesq}.

\begin{lemma}\label{lemma:ibound}
Fix $0<\epsilon<1$. Then there exist some $A>0$ such that the following holds.
Assume $r$ satisfies \eqref{eq:rrange} and take $1\leq i\leq \frac{1}{2} r$. 
Let $\max\{\frac{4B}{\epsilon},i^{4/5}\}<\lambda<\frac{1}{3}r$.
For any $x>A$, the number of $n\in S_r(x)$ such that
$\left|\frac{1}{2}\log\log p_i-i\right|>\lambda$,
is 
\[\ll 
\Phi_r(x)
\exp\left(-\frac{(1-\epsilon)\lambda^2}{2(i+\lambda)}\right).
\]
\end{lemma}

\begin{proof}
The number of $n\in S_r(x)$ such that 
$\frac{1}{2}\log\log p_i>i+\lambda$, by Theorem~\ref{theorem:Tudesq}
is bounded by 
\begin{multline*}
\ll\frac{x}{\log x}\sum_{l=0}^{i-1}
\frac{\left(i+\lambda+B\right)^{l}}{l!}\frac{\left(\mu-(i+\lambda)+B\right)^{r-l}}{(r-l)!}
\ll
\Phi_r(x)
\sum_{l=0}^{i-1}
\binom{r}{l}
\left(\frac{i+\lambda}{\mu}\right)^{l}
\left(1-\frac{i+\lambda}{\mu}\right)^{r-l},
\end{multline*}
where we have used
\[
\left(1+
\frac{B}
{i+\lambda}
\right)^{l}
\left(1+
\frac{B}
{\mu-(i+\lambda)}
\right)^{r-l}
\leq
\exp\left(\frac{iB}
{i+\lambda}
+
\frac{rB}
{\mu-(i+\lambda)}
\right)
\leq
\exp(7B)
.
\]
Using Chernoff bound \cite[Theorem~A.1.12]{probabilistic} and  maximising over $r$ in the range $|\mu-r|<\mu^{2/3}$, we have
\begin{multline*}
\leq
\Phi_r(x)
\exp\left(-\frac{r\mu}{2(i+\lambda)}\left(\frac{i+\lambda}{\mu}-\frac{i}{r}\right)^2\right)
\leq
\Phi_r(x)
\exp\left(-\frac{1-\mu^{-1/3}}{2(i+\lambda)}\left(i+\lambda-\frac{i}{1-\mu^{-1/3}}\right)^2\right)\\
\leq
\Phi_r(x)
\exp\left(-\frac{1-\mu^{-1/3}}{2(i+\lambda)}\left(\lambda-i^{2/3}\right)^2\right)
\leq
\Phi_r(x)
\exp\left(-\frac{(1-\epsilon)\lambda^2}{2(i+\lambda)}\right)
.
\end{multline*}

For the number of $n\in S_r(x)$ satisfying 
$\frac{1}{2}\log\log p_i<i-\lambda$, first note that there are none if $\lambda\geq i$, so assume $\lambda<i$.
Again by Theorem~\ref{theorem:Tudesq}, 
this is bounded by 
\begin{multline*}
\frac{x}{\log x}
\sum_{l=i}^{r}
\frac{\left(i-\lambda+B\right)^{l}}{l!}\frac{\left(\mu-(i-\lambda)+B\right)^{r-l}}{(r-l)!}\\
\ll
\Phi_r(x)
\sum_{l=i}^{r}
\binom{r}{l}
\left(\frac{i-\lambda+B}{\mu}\right)^{l}
\left(1-\frac{i-\lambda+B}{\mu}\right)^{r-l},
\end{multline*}
noting that
\[\left(1+\frac{2B}{\mu-(i-\lambda)-B}\right)^{r-l}
\leq
\exp\left(\frac{2B(r-l)}{\mu-(i-\lambda)-B}\right)
\leq
\exp\left(\frac{2B(r-i)}{\mu-i+\lambda-B}\right)
\leq
\exp(4B).
\]
By Chernoff bound (see for example \cite[Theorem~A.1.9]{probabilistic} with the optimal $\lambda$ given in the remark after) and writing $R:=\frac{r}{\mu}$, so $|R-1|<\mu^{-1/3}$, this is bounded by
\begin{multline*}
\leq\Phi_r(x)
R^i
\left(1-\frac{\lambda-B}{i}\right)^i
\left(1+\frac{i(1-R)+R(\lambda-B)}{r-i}\right)^{r-i}\\
\leq\Phi_r(x)
\exp(-i(1-R))
\exp\left(-(\lambda-B)-\frac{(\lambda-B)^2}{2i}\right)
\exp\left(i(1-R)+R(\lambda-B)\right)\\
\leq\Phi_r(x)
\exp\left(-(\lambda-B)\left(\frac{\lambda-B}{2i}-(R-1)\right)\right)
\leq\Phi_r(x)
\exp\left(-\frac{(1-\epsilon)\lambda^2}{2i}\right)
.\qedhere\end{multline*}
\end{proof}

The theorem is trivial when $\eta >r$, so assume $\eta <r$. Take $\tilde{\eta }=\frac{1}{3}\eta $, so that $\tilde{\eta }<\frac{1}{3}$ and apply Lemma~\ref{lemma:ibound} with $\lambda=\tilde{\eta }^{1/5}\max\{i,\tilde{\eta }\}^{4/5}$.
We get that the number of $n\in S_r(x)$ such that
\[\left|\frac{1}{2}\log\log p_i-i\right|>\eta ^{1/5}\max\{i,\eta \}^{4/5}>\lambda
\text{ for some } i<\frac{1}{3}r,\]
is bounded by
$\ll\Phi_r(x)\exp\left(-\frac{1}{5}\tilde{\eta }\right)$, when $\tilde{\eta }>20B$.
Summing over all $i<\frac{1}{3}r$ gives the bound
$\ll \Phi(x)\exp(-\frac{1}{6}\tilde{\eta })
=\Phi(x)\exp(-\frac{1}{18}\eta)$.

\subsubsection{Proof of Theorem~\ref{theorem:small}\ref{part3}}
Fix $\kappa>\frac{2}{3}$. 
We will show that other than $\ll
\Phi_r(x)\exp\left(-(\log \mu)^{1-\kappa(1+\epsilon)}\right)$
exceptions, we have 
\[\max_{\frac{1}{2}\sqrt{r}<i<\frac{1}{2}r}
\log\log p_i-\log\left(\sum_{j=1}^{i-1}\log p_j\right)
- 2\log \log\log p_i
>(3\kappa-1)\log\log\mu -2.\]

First remove $n\in S_r(x)$ for which 
\[\left|\frac{1}{2}\log\log p_i-i\right|>i^{4/5} \text{ for some } \frac{1}{2}\sqrt{r}<i<\frac{1}{2}r.\] 
Applying Lemma~\ref{lemma:ibound} with $\lambda=i^{4/5}$, for each $i$ we have the bound
$\leq 
\Phi_r(x)\exp\left(-\frac{1}{5}i^{3/5}\right)$.
Summing over $\frac{1}{2}\sqrt{r}<i<\frac{1}{2}r$ gives
$\ll \Phi(x)\exp(-\frac{1}{10}\mu^{3/10})$.

The remaining $n\in S_r(x)$ has
\begin{equation}\label{eq:pi}
\left|\frac{1}{2}\log\log p_i-i\right|<i^{4/5}
\text{ for every }
\frac{1}{2}\sqrt{r}<i<\frac{1}{2}r.
\end{equation}
Let $m=\lceil\frac{1}{2}\sqrt{r}\rceil-1$ and $k=\lfloor \frac{1}{2}r\rfloor-1$, so $p_1\cdots p_k\leq \sqrt{x}$.
We bound the number of $n\in S_r(x)$ for which $p_i< p_{i+1}\leq p_i^{a_i}$ for all $m\leq i\leq k$, where $a_i=(i+1)^2(\log \mu)^{2\kappa}$.
Apply Theorem~\ref{theorem:Tudesq}, with the set $E_1$ containing the primes less than $p_m$ and $E_2$ containing the primes greater than $p_k$, and on numbers up to $\frac{x}{p_m\cdots p_k}$, we get
\[
\ll \frac{x}{\log x}\sum_{\substack{p_ i<p_{i+1}\leq p_i^{a_i}\\m\leq i\leq k\\p_j\not\equiv 3\bmod 4}}
\frac{1}{p_m\cdots p_k}\cdot 
\frac{\left(\mu-\frac{1}{2}\log\log p_{k}+B\right)^{r-k}}{(r-k)!}\cdot
\frac{\left(\frac{1}{2}\log\log p_{m}+B\right)^{m-1}}{(m-1)!}.
\]
Now fixing some $m\leq i\leq k$, by the prime number theorem and partial summation we have
\begin{multline*}
\sum_{\substack{p_{i}<p_{i+1}\leq p_{i}^{a_{i}}\\p_j\not\equiv 3\bmod 4}}
\frac{1}{p_{i+1}}
\frac{\left(\mu-\frac{1}{2}\log\log p_{i+1}+B\right)^{r-i-1}}{(r-i-1)!}\\
\leq
\frac{\left(\mu-\frac{1}{2}\log\log p_{i}+B\right)^{r-i}}{(r-i)!}-\frac{\left(\mu-\frac{1}{2}\log\log p_{i}-\frac{1}{2}\log a_{i}+B\right)^{r-i}}{(r-i)!}+O\left(\exp(-c\sqrt{\log p_i})\right)\\
\leq
\frac{\left(\mu-\frac{1}{2}\log\log p_{i}+B\right)^{r-i}}{(r-i)!}
\left(1-
\left(1-\frac{\frac{1}{2}\log a_{i}}{\mu-\frac{1}{2}\log\log p_{i}+B}
\right)^{r-i}\right)+O\left(\exp(-c\sqrt{\log p_i})\right)\\
\leq 
\frac{\left(\mu-\frac{1}{2}\log\log p_{i}+B\right)^{r-i}}{(r-i)!}
\left(1-a_i^{-\frac{1}{2}+\mu^{-1/3}}
\right)+O\left(\exp(-c\sqrt{\log p_i})\right).\\
\end{multline*}
Applying this repeatedly for $i=k,\ k-2,\dots,\ m$, we have
\[\ll
\Phi_r(x)
\prod_{m\leq i\leq k}
\left(1-\frac{1}{(i(\log \mu)^{\kappa})^{1+\mu^{-1/3}}}\right)
\ll
\Phi_r(x)\exp\left(-(\log \mu)^{1-\kappa(1+\epsilon)}\right).
\]
It follows that other than $\ll
\Phi_r(x)\exp\left(-(\log \mu)^{1-\kappa(1+\epsilon)}\right)$ exceptions, we have
\begin{equation}\label{eq:gap}
p_{i-1}^{i^2(\log\mu)^{2\kappa}}<p_{i}\quad\text{ for some } \frac{1}{2}\sqrt{r}<i<\frac{1}{2}r.
\end{equation}

For the remaining $n\in S_r(x)$, we have
\eqref{eq:pi} and \eqref{eq:gap}
which implies 
\[\max_{\frac{1}{2}\sqrt{r}<i<\frac{1}{2}r}
\log\log p_i-\log\log p_{i-1}
- 2\log \log\log p_i
>2\kappa\log\log\mu -2.\]
It remains to remove $n\in S_r(x)$ for which there exists some $\frac{1}{2}\sqrt{r}\leq i<\frac{1}{2}r$ such that $p_i^{a_i}<p_{i+1}$ and
\[
\sum_{j=1}^{i}\log p_j>(\log \mu)^{1-\kappa}\log p_i.
\]
Rewrite the second condition as $p_i^u<p_1\cdots p_{i-1}$, where $u:=(\log \mu)^{1-\kappa}-1$.
We wish to bound
\begin{multline}\label{eq:smoothbd}
\sum_{\substack{p_i\equiv 1\bmod 4\\ \eqref{eq:pi}}}
\sum_{\substack{p_1<\dots<p_i\\p_1\cdots p_{i-1}>p_i^{u}\\p_j\not\equiv 3\bmod 4}}
\sum_{\substack{p_i^{a_i}<p_{i+1}<\dots<p_r\\
p_{i+1}\cdots p_r<\frac{x}{p_1\cdots p_i}\\p_j\not\equiv 3\bmod 4}} 1
\\
\ll 
\frac{x}{\log x}\sum_{\substack{p_i\equiv 1\bmod 4\\ \eqref{eq:pi}}}
\frac{\left(\mu-\frac{1}{2}\log\log p_i-\frac{1}{2}\log a_i+B\right)^{r-i}}{p_i(r-i)!}
\sum_{\substack{p_1<\dots<p_i\\p_1\cdots p_{i-1}>p_i^{u}\\p_j\not\equiv 3\bmod 4}}
\frac{1}{p_1\cdots p_{i-1}}.
\end{multline}
Fix a given $p_i$ with $\left|\frac{1}{2}\log\log p_i-i\right|<i^{2/3}$ and $p_i^{u}<N<\min\{p_i^i,x\}$, 
by Lemma~\ref{lemma:ubsmooth} we have
\[
\sum_{\substack{p_1<\dots<p_{i-1}<p_i\\N<p_1\cdots p_{i-1}\leq 2N\\p_j\not\equiv 3\bmod 4}}
\frac{1}{p_1\cdots p_{i-1}}
\leq \frac{1}{N}\Psi_{i-1}(2N,p_i)
\ll \frac{v^{-v}}{\log p_i}\cdot\frac{(\frac{1}{2}\log\log p_i)^{i-1}}{(i-1)!},\]
where $v:=\log N/\log p_i$.
To deal with final part of the sum in \eqref{eq:smoothbd},
split $(p_i^u,p_i^{i-1})$ into dyadic intervals of the form $(N,2N]$, then
\begin{multline*}
\sum_{\substack{p_1<\dots<p_i\\p_1\cdots p_{i-1}>p_i^{u}\\p_j\not\equiv 3\bmod 4}}
\frac{1}{p_1\cdots p_{i-1}}
\ll \frac{1}{\log p_i}\cdot\frac{(\frac{1}{2}\log\log p_i)^{i-1}}{(i-1)!}
\sum_{\substack{k\geq 0\\N=2^kp_i^u}}v^{-v}\\
\ll
\frac{(\frac{1}{2}\log\log p_i)^{i-1}}{(i-1)!}\int_{v>u} v^{-v}dv
\ll
u^{-u}\frac{\left(\frac{1}{2}\log \log p_i\right)^{i-1}}{(i-1)!}.
\end{multline*}
Therefore \eqref{eq:smoothbd} becomes
\begin{multline*}
\ll 
\frac{xu^{-u}}{\log x}
\sum_{\substack{p_i\equiv 1\bmod 4\\ \eqref{eq:pi}}}
\frac{1}{p_i}\cdot
\frac{\left(\mu-\frac{1}{2}\log\log p_i\right)^{r-i}}{(r-i)!}
\cdot 
\frac{(\frac{1}{2}\log\log p_i)^{i-1}}{(i-1)!}
\left(1-\frac{\frac{1}{2}\log a_i}{\mu-\frac{1}{2}\log\log p_i+B}\right)^{r-i}
\\
\ll
\frac{x}{\log x}
\cdot
\frac{u^{-u}}{a_i^{1/2-\mu^{-1/3}}}
\sum_{\substack{p_i\equiv 1\bmod 4\\ \eqref{eq:pi}}}
\frac{1}{p_i}\cdot
\frac{\left(\mu-\frac{1}{2}\log\log p_i\right)^{r-i}}{(r-i)!}
\cdot 
\frac{(\frac{1}{2}\log\log p_i)^{i-1}}{(i-1)!}
\ll 
\frac{u^{-u}\Phi_r(x)}{a_i^{1/2-\mu^{-1/3}}}.
\end{multline*}
Summing over $\frac{1}{2}\sqrt{r}<i<\frac{1}{2}r$, 
the total number of such $n$ is 
\[\ll \Phi_r(x)\exp\left(-2(\log \mu)^{1-\kappa}\right)\sum_{\frac{1}{2}r^{1/2}<i<\frac{1}{2}r}\frac{1}{i}
\ll \Phi_r(x)\exp\left(-(\log \mu)^{1-\kappa}\right)
.\]

\section{Equidistribution of Legendre symbol matrices}
\label{sEqui}
We will use the two following propositions from Section 6 of Smith \cite{Smith}.

\begin{prop}
\label{cheb}
Suppose $L/\Q$ is Galois of degree $d$ and $K/\Q$ is an elementary abelian extension, and $\gcd(\Delta_L,\Delta_K)=1$. Let $K_0$ be a quadratic subfield of $K$ with maximal discriminant $|\Delta_{K_0}|$. Let $G:=\Gal(KL/\Q)$ is a $2$-group. Take $F:G\rightarrow [-1,1]$ to be a class function with average $0$ over $G$. Then there exists an absolute constant $c>0$ such that 
\begin{multline*}
\sum_{p\leq x} F\left(\leg{KL/\Q}{p}\right) \log p\\
\ll
x^{\beta}|G|
+x|G|(d^2\log |x\Delta_{K_0}\Delta_L|)^4
\exp\left(\frac{-cd^{-4}\log x}{\sqrt{\log x}+3d\log |\Delta_{K_0}\Delta_L|}\right)
\end{multline*}
for $x\geq 3$, where $\beta$ is the maximal real zero of any Artin $L$-function defined for $G$.
\end{prop}

\begin{proof}
This quickly follows from the Chebotarev Density Theorem, see Proposition~6.5 in Smith \cite{Smith}.
\end{proof}

\begin{prop}
\label{larges}
Let $X_1$ and $X_2$ be disjoint sets of odd primes bounded by $t_1$ and $t_2$ respectively. Then for any $\epsilon>0$, we have
\[\sum_{x_1\in X_1}\left|\sum_{x_2\in X_2}\leg{x_1}{x_2}\right|\ll_{\epsilon}
t_1t_2^{3/4+\epsilon}+t_2t_1^{3/4+\epsilon}.\]
\end{prop}

\begin{proof}
This is an easy consequence of the large sieve inequality stated in the work of Jutila \cite[Lemma~3]{Jutila}, see Proposition~6.6 in Smith \cite{Smith}.
\end{proof}

We shall not work with all squarefree integers simultaneously, but instead work with more restricted sets of squarefree integers that have extra combinatorial structure. In our next definition we define this combinatorial structure, which we call preboxes.

\begin{mydef}
Take a sequence of real numbers
\[
0<s_1<t_1<s_2<t_2<\dots<s_r<t_r.
\]
Take $P, X_1,\dots, X_r$ to be disjoint sets of primes not congruent to $3 \bmod 4$ so that $X_i\subset (s_i,t_i)$. Define $X:=X_1\times\dots\times X_r$. We call the pair $(X, P)$ a prebox.
\end{mydef}

The goal of this section is to prove a weak equidistribution statement regarding matrices of Jacobi symbols associated to each $x \in X$. To make sense of this, we first need to define how we attach a matrix of Jacobi symbols to each $x \in X$, which we shall do now. We will often implicitly identify $\mathbb{F}_2$ with $\{\pm 1\}$ in this section.

\begin{mydef}
\label{dMa}
Let $(X, P)$ be a prebox. Take $\M\subseteq \{(i,j):1\leq i< j\leq r\}$ and $\mathcal{N}\subseteq P\times [r]$. Define $M: X\rightarrow \FF_2^{|\M\sqcup\mathcal{N}|}$ as follows
\[
M(x_1,\dots,x_r):\M\sqcup\mathcal{N}\rightarrow\{\pm 1\}\qquad
M(x_1,\dots,x_r)(\mathbf{m})=
\begin{cases}
\leg{x_i}{x_j}&\text{ if }\mathbf{m}=(i,j)\in\M\\
\leg{p}{x_j}&\text{ if }\mathbf{m}=(p,j)\in\mathcal{N}.
\end{cases}
\]
Denote $\mathcal{N}_{j}:={\{(p,j)\in\mathcal{N}\}}$.
Let $M_j:X_j\rightarrow \FF_2^{|\mathcal{N}_{j}|}$ be the function defined by
\[M_j(x_j):\mathcal{N}_{j}\rightarrow\{\pm 1\}\qquad
M_j(x_j)(p,j)=
\leg{p}{x_j}.
\]
For any $a:\M\sqcup\mathcal{N}\rightarrow\{\pm 1\}$, define 
\[
X(a):=\{x\in X:M(x)=a\},
\]
and $X_j(a,P):=\{x_j\in X_j: M_j(x_j)=a\restriction_{\mathcal{N}_{j}}\}$.
\end{mydef}

Ideally, we would like to prove that $X(a)$ is of the expected size, that is
\[
|X(a)| = \frac{|X|}{2^{|\M| + |\mathcal{N}|}}.
\]
Instead we shall prove a weaker equidistribution statement that allows for permutations of the first few columns.

\begin{mydef}
Let $\mathcal{P}(r)$ denote the set of permutations of $[r]$.
For any $\sigma\in\mathcal{P}(r)$, any prebox $(X,P)$ and any $a:\M\sqcup\mathcal{N}\rightarrow\{\pm 1\}$,
define
\[
X(\sigma,a)=\left\{x\in X:
M(\sigma (x))=a\right\},
\]
where $\sigma(x)=\sigma(x_1,\dots,x_r)=(x_{\sigma(1)},\dots,x_{\sigma(r)})$.
\end{mydef}

Finally, there is the well-known problem of Siegel zeroes that we need to take care of. This prompts the following definition.

\begin{mydef}
For $c>0$, take $\mathcal{S}(c)$ to be the set of squarefree integers $d$ so that 
\[
L(s,\chi_d)=0\text{ for some } 1-\frac{c}{\log d}\leq s\leq 1.
\]
List the elements in $\mathcal{S}(c)$ as $d_1<d_2<\cdots$. By Landau's theorem, fix an absolute $c$ sufficiently small so that
$d_i^2\leq |d_{i+1}|$ for all $i\geq 1$. We say that a prebox $(X,P)$ is Siegel-less above $t$ if the following holds
\[
\left\{x\prod_{p\in \tilde{P}}p>t: x\in X,\ \tilde{P}\subseteq P\right\}\cap \mathcal{S}(c)=\varnothing.
\]
\end{mydef}

We are now ready to prove our first proposition, which shows that $X(a)$ is of the expected size for sufficiently regular prebox $(X, P)$ and sufficiently nice $\M$ and $\mathcal{N}$. It is directly based on Proposition~6.3 in Smith \cite{Smith}.

\begin{prop}
\label{6.3}
Fix positive constants $c_1,\dots, c_6$ such that
$c_2c_3+2c_4+c_5< \frac{1}{4}$, and $c_6>3$. 
Take $\delta>0$ satisfying $2\delta <\frac{1}{4}-c_2c_3-2c_4-c_5$, then the following holds for any large enough $D_1$. Take $1\leq k\leq r$. Suppose $\M\subseteq \{(i,j):1\leq i< j\leq r\}$ and $\mathcal{N}\subseteq P\times\{k+1,\dots,r\}$. Let $a:\M\sqcup\mathcal{N}\rightarrow\{\pm 1\}$.
Let $(X, P)$ be a prebox with parameters $D_1<s_1<t_1<s_2<t_2<\dots<s_r<t_r$ such that 
\[
X_j:=
\{x_j\in (s_j,t_j)\text{ prime}: x_j\equiv 1\bmod 4,\ M_j(x_j)=a\restriction_{\mathcal{N}_{j}}\}\quad\text{if }j>k.
\]
Assume 
\begin{enumerate}
\item $(X,P)$ is Siegel-less above $D_1$;
\item $P\subseteq [t_1]$ and $|P|\leq \log t_i-i$ for all $1\leq i\leq r$;
\item $\log t_{k+1}>\max\{(\log t_1)^{c_6},D_1^{c_1}\}$ if $k<r$, and $\log t_k<t_1^{c_2}$;
\item 
$|X_i|\geq e^it_i(\log t_i)^{-c_3}$ for all $1\leq i\leq r$;
\item $r<D_1^{c_4}$;
\item for each $1\leq i\leq r$, $j_i:= i-1+\lfloor c_5 \log t_i\rfloor$ satisfy $j_1> k$, and
$\log t_{j_i}>(\log t_i)^{c_6}$ if $j_i\leq r$.
\end{enumerate}
Then
\[
\left| |X(a)|-2^{-|\M|}|X|\right|\leq t_1^{-\delta}\cdot 2^{-|\M|}|X|.
\]
\end{prop}

\begin{proof}
Let $\kappa:=c_4+\delta$.
Since $r<D_1^{c_4}$, 
it suffices to show that
\[
\left| |X(a)|-2^{-|\M|}|X|\right|\leq r t_1^{-\kappa}\cdot 2^{-|\M|}|X|.
\]
We proceed by induction on $r$. Define 
\[
X_j(a,x_1):=\left\{x_j\in X_j: \leg{x_1}{x_j}=a(1,j)\text{ if }(1,j)\in\M\right\}.
\]

First consider $(1,j)\in \M$ where $j> k$. Apply Proposition~\ref{cheb} to 
\[
K=\Q(\sqrt{-1},\sqrt{p}:p\in P), \quad L=\Q(\sqrt{x_1})
\]
and 
\[
F:\sigma\mapsto
\begin{cases}
1-2^{-|P|-1} &\text{if } \sigma=\leg{KL/\Q}{x_j}\text{ for some }x_j\in X_j(a,x_1),\\
\hfil -2^{-|P|-1} &\text{otherwise.}
\end{cases}
\]
Notice that $\leg{KL/\Q}{x_j}$ is independent of the choice $x_j\in X_j(a,x_1)$.
By Siegel's theorem, for $D_1$ sufficiently large, we have $1-\beta>D_1^{-c_1/6}$ if $\beta$ is an exceptional real zero of $L(s,\chi_d)$ with $|d|<D_1$. Then 
\[
t_j^{\beta}
< t_j\exp\left(-(\log t_j)^{5/6}\right).
\]
Also 
\[
\log |\Delta_{K_0}\Delta_L|\ll |P|\log t_1 \leq (\log t_1)^2\ll(\log t_j)^{\frac{2}{c_6}}
\]
and 
\[
|\Gal(KL/\Q)|\leq 2^{1+\log t_1}<t_1<\exp\left((\log t_j)^{\frac{1}{c_6}}\right).
\]
Since $c_6>3$, by Proposition~\ref{cheb} for any $u\leq t_j$, we have
\[\left|
\sum_{p\leq u}F\left(\leg{KL/\Q}{p}\right) \log p\right|\leq t_j\exp\left(-(\log t_j)^{1/3}\right).
\]
Then by partial summation
\[\left|\sum_{p\leq t_j}F\left(\leg{KL/\Q}{p}\right)\right|
\leq t_j\exp\left(-(\log t_j)^{1/3}\right).
\]
Combining with similar estimates for $s_j$ and over the field $K/\Q$, we have
\[
\left||X_j(a,x_1)|-\frac{1}{2}|X_j|\right|\leq 4t_j\exp\left(-(\log t_j)^{1/3}\right)<t_1^{-1}|X_j|.
\]

Next consider $(1,j)\in \M$ where $j\leq k$. 
Fix a positive constant $\epsilon$ such that $2\epsilon<\frac{1}{4}-c_2c_3-c_5-2\kappa$. The large sieve in Proposition~\ref{larges} gives
\[
\sum_{x_1\in X_1}\left|\sum_{x_j\in X_j}\leg{x_1}{x_j}\right|\ll_{\epsilon}
t_jt_1^{3/4+\epsilon}.
\]
Since
\[|X_j(a,x_1)|=\frac{1}{2}\sum_{x_j\in X_j}\left(a(1,j)\leg{x_1}{x_j}+1\right)
=\frac{a(1,j)}{2}\sum_{x_j\in X_j}\leg{x_1}{x_j}+\frac{1}{2}|X_j|,\]
for sufficiently large $D_1$, we have 
\[
\sum_{x_1\in X_1}\left|
|X_j(a,x_1)|
-\frac{1}{2}|X_j|
\right|
=\frac{1}{2}\sum_{x_1\in X_1}\left|\sum_{x_j\in X_j}\leg{x_1}{x_j}\right|
\leq t_1^{-\frac{1}{4}+c_2c_3+\epsilon}|X_1||X_j|.\]
Let $B_1:=c_5+\kappa$ and $B_2:=\kappa+\epsilon$, then $B_1+B_2< \frac{1}{4}-c_2c_3-\epsilon$. We deduce that
\[\left|
|X_j(a,x_1)|-\frac{1}{2}|X_j|
\right|<t_1^{-B_2}|X_j| 
\text{ for all }(1,j)\in\M\text{ and }j\leq k
\] 
holds for $x_1\in X$
with at most $kt_1^{-B_1}|X_1|$ exceptions. Call the set of exceptions $X_1^{\text{bad}}(a)$.

We bound the size of the set of exceptions $X^{\text{bad}}(a)=X(a)\cap \pi_1^{-1}(X_1^{\text{bad}}(a))$ in $X$.
First fix some $x_1\in X_1$ and move $x_1$ to $P$. Apply the induction hypothesis to
\[X_2\times X_3\times \dots \times X_k\times X_{k+1}(a,x_1)\times\dots\times X_{r}(a,x_1).\]
Then
$|X(a)\cap \pi_1^{-1}(x_1)|\leq 2^{-|\M|+k+1}\frac{|X|}{|X_1|}$,
so 
\[
|X^{\text{bad}}(a)|
\leq 2^{k+1}k t_1^{-B_1 }\cdot 2^{-|\M|}|X|
<2^{j_1+1}r t_1^{-B_1 }\cdot 2^{-|\M|}|X|
<2r t_1^{-c_5(1-\log 2)-\kappa}\cdot 2^{-|\M|}|X|
\]
which fits into the error term. For $x_1\notin X^{\text{bad}}(a)$, we look at 
\[X_2(a,x_1)\times\dots\times X_r(a,x_1),\]
which has size between $2^{-|\M|}\frac{|X|}{|X_1|}(1\pm (r-1)t_1^{-\kappa})$ by the induction hypothesis. Then
\begin{align*}
|X(a)\setminus X^{\text{bad}}(a)|
&=\sum_{x_1\in X_1\setminus X_1^{\text{bad}}(a)}|X(a)\cap \pi_1^{-1}(x_1)| \\
&=\sum_{x_1\in X_1\setminus X_1^{\text{bad}}(a)}|(X_2(a,x_1)\times\dots\times X_r(a,x_1))(a)|\end{align*}
which lies between 
\[\left(1\pm t_1^{-1}\right)^r\left(1\pm t_1^{-B_2}\right)^k
\left(1
\pm(r-1)t_1^{-\kappa}\right)\cdot 
2^{-|\M|}|X|
.\]
Since $r<t_1^{c_4}<t_1^{\kappa}$, $1-c_4>\kappa$ and $B_2>\kappa$,  we have
\begin{multline*}
\left(1+ t_1^{-1}\right)^r\left(1+ t_1^{-B_2}\right)^k
\frac{1
+(r-1)t_1^{-\kappa}}{1
+ rt_1^{-\kappa}}
=
\left(1+ t_1^{-1}\right)^r\left(1+ t_1^{-B_2}\right)^k
\left(1-
\frac{
t_1^{-\kappa}}{1
+ rt_1^{-\kappa}}\right)\\
<\exp\left(rt_1^{-1}+kt_1^{-B_2}-\frac{1}{2}t_1^{-\kappa}\right)<
\exp\left(t_1^{-(1-c_4)}
+c_5t_1^{-B_2}\log t_1
-\frac{1}{2}t_1^{-\kappa}\right)
<1
\end{multline*}
and similarly
\begin{multline*}
\left(1- t_1^{-1}\right)^r\left(1- t_1^{-B_2}\right)^k
\frac{1
+(r-1)t_1^{-\kappa}}{1
- rt_1^{-\kappa}}
=
\left(1- t_1^{-1}\right)^r\left(1- t_1^{-B_2}\right)^k
\left(1+
\frac{
t_1^{-\kappa}}{1
- rt_1^{-\kappa}}\right)\\
>\exp\left(-rt_1^{-1}-kt_1^{-B_2}+\frac{1}{2}t_1^{-\kappa}\right)>
\exp\left(-t_1^{-(1-c_4)}
-c_5t_1^{-B_2}\log t_1
+\frac{1}{2}t_1^{-\kappa}\right)
>1.
\end{multline*}
This completes the inductive step.
\end{proof}

The condition $\mathcal{N}\subseteq P\times\{k+1,\dots,r\}$ in Proposition~\ref{6.3} turns out to be too restrictive for us. It is however not so straightforward to remove this condition. Hence we shall only prove a weaker equidistribution statement that allows for permutations of the first few columns. This weaker equidistribution statement will fall as a consequence of Proposition~\ref{6.3} and the following combinatorial proposition, which is Proposition~6.7 of Smith \cite{Smith}.

\begin{prop}
\label{6.7}
Let $(X,P)$ be a prebox. 
Let $\M= \{(i,j):1\leq i< j\leq r\}$ and $\mathcal{N}= P\times [r]$.
Take $0\leq k_0\leq k_1\leq k_2\leq r$ so that
\[2^{|P|+k_0+1}k_1^2<k_2.\]
Let $\sigma\in\mathcal{P}(r)$.
Define
\[
S(\sigma):=
\left\{(i,j)\in\M:(\sigma(i),\sigma(j))\in ([k_0]\times[k_1])\cup ([k_1]\times[k_0])\right\}
\sqcup \left\{(p,j)\in\mathcal{N}:\sigma(j)\in [k_1] \right\}.\]
Let $m:=|S(\sigma)|=k_1|P|+\frac{1}{2}k_0(k_0-1)+k_0(k_1-k_0)$. If $a:\M\sqcup\mathcal{N}\rightarrow\{\pm 1\}$, we put
\[
X_S(\sigma,a):=\left\{x\in X:
M(\sigma (x))\restriction_{S(\sigma)}=a\restriction_{S(\sigma)}\right\}.
\]
For any $x\in X$, define 
\[
W(x,a):=\{\sigma\in\mathcal{P}(k_2): x\in X_S(\sigma,a)\}=
\{\sigma\in\mathcal{P}(k_2): M(\sigma(x))\restriction_{S(\sigma)}=a\restriction_{S(\sigma)}\}.\]
Then we have
\[
\sum_{a\in\FF_2^{\M\sqcup\mathcal{N}}}\left|
|W(x,a)|-2^{-m}\cdot k_2!\right|
\leq \left(\frac{2^{|P|+k_0+1}}{k_2}\right)^{1/2}k_1\cdot 2^{-m}\cdot k_2!.
\]
\end{prop}

\begin{proof}
Fix some $x\in X$ and write $W(a):=W(x,a)$.
We will show that
\[\sum_{a\in\FF_2^{\M\sqcup\mathcal{N}}}\left(|W(a)|-2^{-m}\cdot k_2!\right)^2
\leq \frac{2^{|P|+k_0+1} k_1^2}{k_2}\cdot 2^{-2m+|\M\sqcup\mathcal{N}|}(k_2!)^2
,\]
then the proposition follows from the Cauchy-Schwarz inequality.

The average of $|W(a)|$ is $2^{-m}\cdot k_2!$, since $|\mathcal{P}(k_2)|=k_2!$ and there are $m$ Legendre symbol conditions to satisfy.
Now 
\[
|W(a)|^2=|\{(\sigma_1,\sigma_2)\in\mathcal{P}(k_2)\times\mathcal{P}(k_2): M(\sigma_1(x))\restriction_{S(\sigma_1)}=a\restriction_{S(\sigma_1)},\ M(\sigma_2(x))\restriction_{S(\sigma_2)}=a\restriction_{S(\sigma_2)}\}|.
\]
We have $\sum_{a\in\FF_2^{\M\sqcup\mathcal{N}}}|W(a)|^2=\sum_{\sigma_1,\sigma_2\in\mathcal{P}(k_2)}|W(\sigma_1,\sigma_2)|$, where
\[
W(\sigma_1,\sigma_2):=\{a\in\FF_2^{\M\sqcup\mathcal{N}}: 
M(\sigma_1(x))\restriction_{S(\sigma_1)}=a\restriction_{S(\sigma_1)},\ M(\sigma_2(x))\restriction_{S(\sigma_2)}=a\restriction_{S(\sigma_2)}\}.
\]
We fix some $\sigma_1,\sigma_2\in\mathcal{P}(k_2)$ and bound $|W(\sigma_1,\sigma_2)|$.
Let 
$
d:=|\{i\in[k_2]:\sigma_1(i)\leq k_1,\ \sigma_2(i)\leq k_1\}|$.
We have
\begin{multline*}
|S(\sigma_1)\cap S(\sigma_2)|
=|\left\{(i,j)\in\M:
(\sigma_1(i),\sigma_1(j)),\ (\sigma_2(i),\sigma_2(j))\in ([k_0]\times[k_1])\cup ([k_1]\times[k_0])\right\}|
\\
+|\left\{(p,j)\in\mathcal{N}:\sigma_1(j),\ \sigma_2(j)\in [k_1] \right\}|
\leq d(|P|+k_0).
\end{multline*}
Therefore the conditions fixes at least 
$2m-d(|P|+k_0)$ arguments of $a\in W(\sigma_1,\sigma_2)$.
Then 
\[|W(\sigma_1,\sigma_2)|\leq 2^{-2m+d(|P|+k_0)+|\M\sqcup\mathcal{N}|}.\]

Given some $d\leq k_1$, we bound the number of $(\sigma_1,\sigma_2)\in\mathcal{P}(k_2)\times \mathcal{P}(k_2)$ that gives the same $d$. 
There are $\binom{k_2}{d}$ ways to pick the indices that map to $[k_1]$ under $\sigma_1$ and $\sigma_2$. 
Then there are at most $(\frac{k_1!}{(k_1-d)!}(k_2-d)!)^2$ ways to pick a pair of $(\sigma_1,\sigma_2)$ in such a way.
Hence the total number is bounded by 
\[\binom{k_2}{d}\left(\frac{k_1!}{(k_1-d)!}(k_2-d)!\right)^2\leq 
(k_2!)^2\left(\frac{k_1!}{(k_1-d)!}\right)^2\frac{ (k_2-d)!}{k_2!}
\leq (k_2!)^2\left(\frac{k_1^2}{k_2}\right)^d
.\]
The average of $|W(a)|^2$ is bounded by \[(k_2!)^2\sum_{d\geq 0}\left(\frac{k_1^2}{k_2}\right)^d \cdot 2^{-2m+d(|P|+k_0)}
=\frac{k_2}{k_2-2^{|P|+k_0}k_1^2}\cdot 2^{-2m}(k_2!)^2.
\]
Then the variance of $|W(a)|$ is bounded by
\begin{align*}
\frac{k_2}{k_2-2^{|P|+k_0}k_1^2}\cdot 2^{-2m}(k_2!)^2-(2^{-m}\cdot k_2!)^2
&=\frac{2^{|P|+k_0}k_1^2}{k_2-2^{|P|+k_0}k_1^2}\cdot 2^{-2m}(k_2!)^2 \\
&\leq \frac{2^{|P|+k_0+1}k_1^2}{k_2}\cdot 2^{-2m}(k_2!)^2
.\end{align*}
Multiplying by $2^{|\M\sqcup\mathcal{N}|}$ gives the required estimate.
\end{proof}

We are now ready to prove our weak equidistribution result for $|X(a)|$, which is very similar to Theorem~6.4 in Smith \cite{Smith}.

\begin{theorem}
\label{6.4}
Take positive constants $c_1,\dots,c_8$, where $c_2c_3+2c_4+c_5< \frac{1}{4}$, $c_6>3$ and $c_8<c_7<\frac{1}{2}$.
Let $(X, P)$ be a prebox and suppose that for all $1\leq j\leq r$
\[X_j:=
\{x_j\in (s_j,t_j)\text{ prime}: x_j\equiv 1\bmod 4\}.\]
The following holds for any large enough $A$ and $D_1$.
Choose integers $0\leq k_0\leq k_1<k_2<r$ and assume $t_{k_0+1}>D_1$ and $k_2>A$.
Assume 
\begin{enumerate}
\item $\log k_1< c_8\log k_2$;
\item $(|P|+k_0)\log 2< (1-2c_7)\log k_2$.
\end{enumerate}
To apply Proposition~\ref{6.3}, assume further
\begin{enumerate}
\item $(X,P)$ is Siegel-less above $D_1$;
\item $P\subseteq [t_{k_0+1}]$ and $|P|\leq \log t_i-i$ for all $k_0 < i\leq r$;
\item $\log t_{k_1+1}>\max\{(\log t_1)^{c_6},D_1^{c_1}\}$ if $k_1<r$, and $\log t_{k_1}<t_1^{c_2}$;
\item 
$|X_i|\geq 2^{|P|}e^{i}k_2^{c_7}t_i(\log t_i)^{-c_3}$ for all $k_0 < i\leq r$;
\item $r<D_1^{c_4}$;
\item for each $k_0 < i\leq r$, $j_i:= i-1+\lfloor c_5 \log t_i\rfloor$ satisfy $j_{k_0+1}> k_1$, and
$\log t_{j_i}>(\log t_i)^{c_6}$ if $j_i\leq r$.
\end{enumerate}

Take $\delta_1<c_7-c_8$ and $2\delta_2 <\frac{1}{4}-c_2c_3-3c_4-c_5$.
Then for any $\M$ and $\mathcal{N}$, we have
\[\sum_{a\in\FF_2^{\M\sqcup\mathcal{N}}}
\left|2^{-|\M\sqcup\mathcal{N}|}\cdot k_2!\cdot |X|-\sum_{\sigma\in\mathcal{P}(k_2)}|X(\sigma,a)|\right|
\leq (k_2^{-\delta_1}+t_{k_0 + 1}^{-\delta_2})\cdot k_2!\cdot |X|.\]
\end{theorem}

\begin{proof}
Without loss of generality assume $\M= \{(i,j):1\leq i < j\leq r\}$ and $\mathcal{N}= P\times [r]$ and $X_i=\{x_i\}$ for $i\leq k_0$. 

Let $m:=k_1|P|+\frac{1}{2}k_0(k_0-1)+k_0(k_1-k_0)$ as in Proposition~\ref{6.7}.
Apply the triangle inequality to the sum we wish to bound 
\begin{multline}\label{6.4sum}
2^{-|\M\sqcup\mathcal{N}|}\sum_{a\in\FF_2^{\M\sqcup\mathcal{N}}}
\left|k_2!\cdot |X|- 2^m\sum_{\sigma\in\mathcal{P}(k_2)}|X_S(\sigma,a)|\right|\\+
2^{-|\M\sqcup\mathcal{N}|+m}\sum_{\sigma\in\mathcal{P}(k_2)}\sum_{a\in\FF_2^{\M\sqcup\mathcal{N}}}
\left||X_S(\sigma,a)|-2^{|\M\sqcup\mathcal{N}|-m}|X(\sigma,a)|\right|
.
\end{multline}
For the first sum in \eqref{6.4sum}, noting that
\[
\sum_{x\in X}|W(x,a)|=\sum_{\sigma\in\mathcal{P}(k_2)}|X_S(\sigma,a)|,
\]
we obtain by Proposition~\ref{6.7} an upper bound 
\[\left(\frac{2^{|P|+k_0+1}}{k_2}\right)^{1/2} k_1\cdot k_2!\cdot |X|<k_2^{-\delta_1}\cdot k_2!\cdot |X|.\]

Now consider the second sum of \eqref{6.4sum}, for each $\sigma\in\mathcal{P}(k_2)$, we can partition $X$ into $2^m$ sets according to $\tilde{a}:S(\sigma)\rightarrow \{\pm 1\}$ as follows
\[X_S(\sigma,\tilde{a})=
\{x_1\}\times\dots\times\{x_{k_0}\}\times 
X_{k_0+1}(\tilde{a},\tilde{P})\times \dots \times X_{k_1}(\tilde{a},\tilde{P})\times 
X_{k_1+1}\times\dots\times X_r,\]
where $\tilde{P}=\{x_1\}\cup\dots\cup\{x_{k_0}\}\cup P$.

We first bound the contribution of $a\in\mathcal{P}(k_2)$ with $|X_i(a,\tilde{P})|< 2^{-|\tilde{P}|}k_2^{-c_7}|X_i|$ for some $k_0< i\leq k_2$ in the sum. For each $\sigma\in\mathcal{P}(k_2)$ and $k_0<i\leq k_1$, we have the upper bound 
\[\sum_{\tilde{a}:|X_i(\tilde{a},\tilde{P})|\leq 2^{-|\tilde{P}|}k_2^{-c_7} |X_i| }\left|X_S(\sigma,\tilde{a})\right|
\leq k_2^{-c_7}|X|.\]
For each $\tilde{a}$, there are $2^{|\M\sqcup\mathcal{N}|-m}$ many $a$ satisfying $a\restriction_{S(\sigma)}=\tilde{a}$, so
the contribution of such $a$ is bounded by
\[
\sum_{\sigma\in\mathcal{P}(k_2)}\sum_{k_0<i\leq k_1}\sum_{\tilde{a}:|X_i(\tilde{a},\tilde{P})|\leq 2^{-|\tilde{P}|}\cdot k_2^{-c_7}\cdot |X_i| } |X_S(\sigma,\tilde{a})|\\
\leq k_1 k_2^{-c_7}\cdot k_2!\cdot |X|
< k_2^{-\delta_1}\cdot k_2! \cdot |X|
.
\]

For the remaining terms we have $|X_i(a,\tilde{P})|< 2^{-|\tilde{P}|} k_2^{-c_7} |X_i|$ for all $k_0< i\leq r$. Bound each summand by Proposition~\ref{6.3}
\[\left||X_S(\sigma,a)|-2^{|\M\sqcup\mathcal{N}|-m} |X(\sigma,a)|\right|
\leq
t_{k_0 + 1}^{-\delta_2}|X_S(\sigma,a)|,
\]
then summing over $\sigma$ and $a$ gives the required estimate.
\end{proof}

There is a final technical proposition that will be of key importance in our next section. We will now state and prove it.

\begin{prop}
\label{pPointSub}

Fix positive constants $c_1,\dots, c_6$ such that
$c_2c_3+2c_4+c_5< \frac{1}{4}$, and $c_6>3$. 
Take $\delta>0$ satisfying $2\delta <\frac{1}{4}-c_2c_3-2c_4-c_5$, then the following holds for any large enough $D_1$. Take $1\leq k\leq r$. Suppose $\M\subseteq \{(i,j):1\leq i< j\leq r\}$ and $\mathcal{N}\subseteq P\times\{k+1,\dots,r\}$. Let $a:\M\sqcup\mathcal{N}\rightarrow\{\pm 1\}$. Let $U, V \subseteq [r]$ be disjoint subsets such that $U \cup V = [l]$ for some $l$. 
Let $(X, P)$ be a prebox with parameters $D_1<s_1<t_1<s_2<t_2<\dots<s_r<t_r$ such that 
\[
X_j:=
\{x_j\in (s_j,t_j)\text{ prime}: x_j\equiv 1\bmod 4,\ M_j(x_j)=a\restriction_{\mathcal{N}_{j}}\}\quad\text{if }j>k+|U|.
\]

Further assume 
\begin{enumerate}
\item $(X,P)$ is Siegel-less above $D_1$;
\item $P\subseteq [t_1]$ and $|P|\leq \log t_i-i$ for all $1\leq i\leq r$;
\item $\log t_{k+1}>\max\{(\log t_1)^{c_6},D_1^{c_1}\}$ if $k<r$, and $\log t_{k+|U|}<t_1^{c_2}$;
\item 
$|X_i|\geq e^it_i(\log t_i)^{-c_3}$ for all $1\leq i\leq r$;
\item $r<D_1^{c_4}$;
\item for each $1\leq i\leq r$, $j_i:= i-1+\lfloor c_5 \log t_i\rfloor$ satisfy $j_1> k+|U|$, and
$\log t_{j_i}>(\log t_i)^{c_6}$ if $j_i\leq r$;
\item $\log \log t_u > \frac{1}{5} \log \log t_r$ for all $u \in U$.
\end{enumerate}

We say that $Q \in \pi_V(X)$ is poor if there is $u \in U$ such that
\[
\left||X_u(a, Q)| - \frac{|X_u(a)|}{2^{|V|}}\right| > t_1^{-2c_4 - 2\delta} |V| |X_u|.
\]
Then
\[
\sum_{Q \in \pi_V(X) \textup{ poor }} |X(a, Q)| \leq  r \cdot t_1^{-(c_4 + \delta)} \cdot 2^{-|\M|} |X|.
\]
\end{prop}

\begin{proof}
We proceed by induction on $|V|$. The case $|V| = 0$ is trivial. Let $v$ be the smallest element in $V$.  Then we clearly have that $v \leq 1 + |U| \leq k + |U|$. Fix some $x \in X_v$. Put $B_1 := c_4 + c_5 + \delta$ and $B_2 := 2c_4 + 2\delta$. Following the proof of Proposition~\ref{6.3} we get that
\[
\left||X_j(a, x)| - \frac{1}{2} |X_j| \right| < t_1^{-B_2} |X_j| \text{ for all } 1 \leq j \leq k + |U| \text{ with } j \neq v
\] 
holds for $x \in X_v$ with at most $(k + |U|) t_1^{-B_1} |X_v|$ exceptions, while for $j > k + |U|$ we always get
\[
\left||X_j(a, x)| - \frac{1}{2} |X_j| \right| < t_j^{-1} |X_j|.
\]
Just as in the proof of Proposition~\ref{6.3} define $X_v^{\text{bad}}(a)$ to be the set of exceptions. We split the sum in the proposition as
\begin{align}
\label{eSplitSub}
\sum_{Q \in \pi_V(X) \textup{ poor }} |X(a, Q)| &= \sum_{\substack{Q \in \pi_V(X) \textup{ poor } \\ \pi_v(Q) \not \in X_v^{\text{bad}}(a)}} |X(a, Q)| + \sum_{\substack{Q \in \pi_V(X) \textup{ poor } \\ \pi_v(Q) \in X_v^{\text{bad}}(a)}} |X(a, Q)| \nonumber \\
&\leq \sum_{\substack{Q \in \pi_V(X) \textup{ poor } \\ \pi_v(Q) \not \in X_v^{\text{bad}}(a)}} |X(a, Q)| + \sum_{\substack{Q \in \pi_V(X) \\ \pi_v(Q) \in X_v^{\text{bad}}(a)}} |X(a, Q)|.
\end{align}
We first treat the latter sum in equation~\eqref{eSplitSub}. In the case $v = 1$, we apply Proposition~\ref{6.3} to the prebox
\[
(X_2 \times \dots \times X_{k + |U|} \times X_{k + |U| + 1}(a, x) \times \dots \times X_r(a, x), P \cup \{x\})
\]
for $x \in X_1^{\text{bad}}(a)$ and the natural restrictions of $a$, $U$, $V$, $\M$ and $\mathcal{N}$. Then the latter sum is bounded by
\[
\sum_{\substack{Q \in \pi_V(X) \\ \pi_1(Q) \in X_1^{\text{bad}}(a)}} |X(a, Q)| = \sum_{x \in  X_1^{\text{bad}}(a)} |X(a) \cap \pi_1^{-1}(x)| \leq |X_1^{\text{bad}}(a)| \cdot 2^{-|\M| + k + |U| + 1} \frac{|X|}{|X_1|}.
\]
A small computation shows that this is at most
\[
\frac{1}{2} \cdot r \cdot t_1^{-(c_4 + \delta)} \cdot 2^{-|\M|} |X|
\]
for sufficiently large $D_1$. Now suppose that $v \neq 1$ so that $1 \in U$. Then apply Proposition~\ref{6.3} with $k = r - 1$, the prebox
\[
(X_1 \times \dots \times X_{v - 1} \times X_{v + 1} \times \dots \times X_r, \varnothing)
\]
and the natural restrictions of $a$, $U$, $V$, $\M$ and $\mathcal{N}$. Crucially, we have that this choice of $k$ satisfies the requirements of Proposition~\ref{6.3} for sufficiently large $D_1$ due to our assumptions $c_5 > 5$ and $\log \log t_1 > \frac{1}{5} \log \log t_r$. Then a similar computation shows that the latter sum is again at most
\[
\frac{1}{2} \cdot r \cdot t_1^{-(c_4 + \delta)} \cdot 2^{-|\M|} |X|.
\]
It remains to bound the former sum in equation~\eqref{eSplitSub}. We first treat the case $v = 1$. Take a poor $Q \in \pi_V(X)$ with $x := \pi_1(Q) \not \in X_1^{\text{bad}}(a)$. Then we claim that $\pi_{V - \{1\}}(Q)$ is poor for the prebox
\[
(X_2(a, x) \times \dots \times X_r(a, x), P \cup \{x\}).
\]
Suppose that $\pi_{V - \{1\}}(Q)$ is not poor. Then we get for all $u \in U$ that
\[
\left||X_u(a, Q)| - \frac{|X_u(a, x)|}{2^{|V| - 1}}\right| \leq t_2^{-B_2} (|V| - 1) |X_u|.
\]
But from this we deduce that for all $u \in U$
\begin{align*}
\left||X_u(a, Q)| - \frac{|X_u(a)|}{2^{|V|}}\right| &\leq \left||X_u(a, Q)| - \frac{|X_u(a, x)|}{2^{|V| - 1}}\right| + \left|\frac{|X_u(a, x)|}{2^{|V| - 1}} - \frac{|X_u(a)|}{2^{|V|}}\right| \\
&\leq t_2^{-B_2} (|V| - 1) |X_u| + t_1^{-B_2} |X_u| \leq t_1^{-B_2} |V| |X_u|,
\end{align*}
establishing the claim. Now we can easily bound equation~\eqref{eSplitSub} using the induction hypothesis. Finally we deal with the case that $v \neq 1$ so that $1 \in U$. In this case we apply the induction hypothesis to the prebox
\[
(X_1 \times \dots \times X_{v - 1} \times X_{v + 1} \times \dots \times X_r, \varnothing)
\]
shifting $k + |U|$ all the way up to $r - 1$.
\end{proof}

As alluded to earlier, the squarefree integers play a crucial role in our analysis. It turns out to be more convenient to work with squarefree integers with a fixed number of prime divisors, and this naturally leads to the following definition.

We now define special preboxes that we call boxes. These boxes provide a natural way to study distributional properties $S_r(N)$ as we shall see in the coming proposition, which is based on Proposition~6.9 in Smith \cite{Smith}.

\begin{mydef}
Suppose $0 \leq k\leq r$. For any $\mathbf{t}=(p_1,\dots,p_k,s_{k+1},\dots,s_r)$ such that
\begin{enumerate}
\item $p_1<p_2<\dots<p_k<D_1$ is a sequence of primes not congruent to $3\bmod 4$,
\item $D_1<s_{k+1}<t_{k+1}<s_{k+2}<t_{k+2}<\dots <s_r<t_r$ is a sequence of real numbers where
\[t_i=\left(1+\frac{1}{e^{i-k}\log D_1}\right)s_i.\]
\end{enumerate}
Define \[X(\mathbf{t}):=X_1\times \dots\times X_r\] with
\[X_i:=
\begin{cases}
\hfil\{p_i\}&\text{if }i\leq k,\\
\{p\in (s_j,t_j)\text{ prime}: p\equiv 1\bmod 4\}&\text{if }i>k.\end{cases}\]

We call $X$ a box if $X=X(\mathbf{t})$ for some $\mathbf{t}$.
There is a bijection from $X$ to a subset of $S_r(N)$. By abuse of notation, denote this subset by $X$. 
\end{mydef}

\begin{theorem}
\label{6.9}
Take $N\geq D_1\geq 3$ with $\log N\geq (\log D_1)^2$. 
Let $W\subseteq S_r(N)$ be a set of comfortably spaced elements above $D_1$ such that
\[\left||W|-\Phi_r(N)\right|<\epsilon\Phi_r(N)\]
for some constant $\epsilon>0$.
Let $V\subseteq S_r(N)$
and suppose that there exists some constant $\delta>0$ such that
\[\left||V\cap X|-\delta|X|\right|<\epsilon |X|\]
for any box $X\subseteq S_r(N)$ satisfying $X\cap W\neq\varnothing$.
Then 
\[|V|-\delta\Phi_r(N)\ll\left(\epsilon+\frac{1}{\log D_1}\right)\Phi_r(N).\]
\end{theorem}

\begin{proof}
Define $\mathcal{T}_k=\{\mathbf{t}:X(\mathbf{t})\cap W\neq \varnothing\}$.
Our aim is to estimate $|V|$ in terms of
\[\int_{\mathcal{T}_k}|V\cap X(\mathbf{t})|\frac{dp_1\cdots dp_kds_{k+1}\cdots ds_r}{s_{k+1}\cdots s_r},\]
where $dp_i$ is $1$ if $p_i\equiv 1\bmod 4$ is prime and $0$ otherwise.

Consider $n=(q_1,\dots,q_r)\in S_r(N)$ with exactly $k$ prime factors less than $D_1$. 
Then $n\in X(\mathbf{t})$ if and only if
$q_i=p_i$ for $1\leq i\leq k$ and 
\[s_i<q_i<\left(1+\frac{1}{e^{i-k}\log D_1}\right)s_i \text{ for }k<i\leq r.\]
If $n\in W$ and
\begin{equation}\label{6.1}
n\prod_{i=k+1}^r \left(1+\frac{1}{e^{i-k}\log D_1}\right)<N,
\end{equation}
then 
\[\begin{split}
\int_{\substack{\mathbf{t}\in\mathcal{T}_k:\\n\in X(\mathbf{t})}}\frac{dp_1\cdots dp_kds_{k+1}\cdots ds_r}{s_{k+1}\cdots s_r}
&=\int_{q_{k+1}\left(1+\frac{1}{e\log D_1}\right)^{-1}}^{q_{k+1}}\cdots\int_{q_r\left(1+\frac{1}{e^{r-k}\log D_1}\right)^{-1}}^{q_r}\frac{ds_{k+1}\cdots ds_r}{s_{k+1}\dots s_r}\\
&=\prod_{i=k+1}^r \log\left(1+\frac{1}{e^{i-k}\log D_1}\right)
.\end{split}\]
If \eqref{6.1} does not hold or $n\notin W$, 
then 
\[\int_{\substack{\mathbf{t}\in\mathcal{T}_k:\\n\in X(\mathbf{t})}}\frac{dp_1\cdots dp_kds_{k+1}\cdots ds_r}{s_{k+1}\cdots s_r}
\leq\prod_{i=k+1}^r \log\left(1+\frac{1}{e^{i-k}\log D_1}\right)
.\]

There exists some constant $C>0$ such that any $n$ that does not satisfy \eqref{6.1} lies in 
\[N\left(1-\frac{C}{\log D_1}\right)
\leq N\prod_{i=k+1}^r\left(1+\frac{1}{e^{i-k}\log D_1}\right)^{-1}\leq n\leq N.\]
The number of such $n$ in $S_r(N,D)$ is bounded by 
\[
\Phi_r(N)-\Phi_r\left(N\left(1-\frac{C}{\log D_1}\right)\right)
\ll
\frac{\Phi_r(N)}{\log D_1},
\]
using estimates from the Selberg--Sathe Theorem~\cite{Selberg}.

Then
\[\sum_{k=0}^{\infty}\prod_{i=k+1}^r \log\left(1+\frac{1}{e^{i-k}\log D_1}\right)^{-1}\int_{\mathcal{T}_k}|V\cap X(\mathbf{t})|\frac{dp_1\cdots dp_kds_{k+1}\cdots ds_r}{s_{k+1}\cdots s_r}
\]
is bounded above by $|V|$ and below by
\[|V\cap W|+O\left(\frac{\Phi_r(N)}{\log D_1}\right)=|V|+O\left(\left(\epsilon+\frac{1}{\log D_1}\right)\Phi_r(N)\right).\]
Similarly
\begin{multline*}
\sum_{k=0}^{\infty}
\prod_{i=k+1}^r 
\log\left(1+\frac{1}{e^{i-k}\log D_1}\right)^{-1}
\int_{\mathcal{T}_k}|X(\mathbf{t})|\frac{dp_1\cdots dp_kds_{k+1}\cdots ds_r}
{s_{k+1}\cdots s_r}\\
=\Phi_r(N)
\left(1+O\left(\epsilon+\frac{1}{\log D_1}\right)\right).
\end{multline*}
The result follows from the estimate $|V\cap X|=(\delta+O(\epsilon))|X|$.
\end{proof}

Our next proposition deals with boxes that are not Siegel-less. It is directly based on Proposition~6.10 in Smith \cite{Smith}.

\begin{theorem}
\label{6.10}
Let $d_1,d_2,\dots$ be a sequence of distinct squarefree integers greater than $D_1$ satisfying $d_i^2<d_{i+1}$.
Take $N\geq D_1\geq 3$ satisfying $\log N\geq (\log D_1)^2$.
Define
\[V_i:=\bigcup \{X\subseteq S_r(N)\text{ a box}: d_i\mid x\text{ for some }x\in X\}.\]
Then 
\[|\cup_{i\geq 1} V_i|\ll\frac{\Phi_r(N)}{\log D_1}.\]
\end{theorem}

\begin{proof}
Suppose we have some box $X\subseteq V_i$ and $d_i=p_1\cdots p_m$.
For any element $x\in X$, there are prime factors $q_1,\dots,q_m$ of $x$ such that
\[q_i=p_i \text{ if } p_i<D_1,\text{ and }\quad \frac{1}{2}p_i<q_i<2p_i\text{ if } p_i\geq D_1.\]

If $d_i<N^{2/3}$, then there exists some constant $C>0$ such that 
\[\begin{split}
|V_i|&
\leq
\Phi_{r-m}\left(\frac{2^m N}{d_i}\right)
\cdot\prod_{p_i\geq D_1}\left|\left\{q_i\text{ prime}:\frac{1}{2}p_i<q_i<2p_i,\ q_i\equiv 1\bmod 4\right\}\right|\\
&\leq \Phi_r(N)\cdot \frac{C^m }{d_i}\prod_{p_i\geq D_1}\frac{p_i}{\log p_i}
\ll\frac{\Phi_r(N)}{\log d_i}.
\end{split}\]
Notice that $d_i>D_1^{2^{i-1}}$. 
Then 
\[\left|\bigcup_{d_i<N^{2/3}}V_i\right|
\ll\Phi_r(N)\sum_{i\geq 1}\frac{1}{2^{i-2}\log D_1}
\ll\frac{\Phi_r(N)}{\log D_1}.\]

If $d_i\geq N^{2/3}$, then $d_{i+1}\geq N^{4/3}>N$. Therefore there is at most one $i$ such that $d_i\geq N^{2/3}$ and $V_i$ is not empty. Then for sufficiently large $D_1$
\begin{align*}
|V_i|&\leq
\left|\left\{x\in S_r(N):d_i\mid x\right\}\right|\cdot\prod_{p_i\geq D_1}\left|\left\{q_i\text{ prime}:\frac{1}{2}p_i<q_i<2p_i,\ q_i\equiv 1\bmod 4\right\}\right|\\
&\leq
\frac{N}{d_i}\prod_{p_i\geq D_1}\frac{p_i}{\log p_i}
\leq \frac{N}{\log d_i}
\ll \frac{N}{\log N},
\end{align*}
which fits into the error bound.
\end{proof}

\begin{mydef}
Fix some constants $c_{9}, c_{10}>0$.
We call a box $X$ of $S_r(N)$ acceptable if
it 
\begin{enumerate}
\item contains a comfortably spaced element above $D_1=\exp\left(\left(\frac{1}{2}\log\log N\right)^{c_{9}}\right)$,
\item contains a $(c_{10}\log\log\log N)$-regular element, and
\item is Siegel-less above $D_1$.
\end{enumerate}
\end{mydef}

Let $\Msym_n$ denote the set of $n\times n$ symmetric matrices over $\FF_2$. Given any integer $x$, let $p_1<\dots<p_n$ be the distinct prime factors of $x$, and call the matrix $(c_{ij})_{2\leq i,j\leq n}\in \Msym_{n-1}$ defined by
\[(-1)^{c_{ij}}=
\begin{cases}
\hfil \leg{p_i}{p_j}&\text{ if }i\neq j\\
\prod_{l\neq j} \leg{p_l}{p_j}&\text{ if }i=j
\end{cases}\]
the Legendre matrix of $x$. We are now ready to reprove a well-known result due to Fouvry and Kl\"uners \cite{FK2}. Note that unlike the work of Fouvry and Kl\"uners our theorem has the benefit of providing an error term.

\begin{theorem}
\label{tFK}
There exists a constant $c>0$, such that
\[\left|\frac{|\{D<N:D\in\mathcal{D} ,\ \rk_4 D=k\}|}{|\mathcal{D}|}-\lim_{n\rightarrow\infty} P(k|n)\right|\ll (\log\log N)^{-c},\]
where 
\[P(k|n):=\frac{|\{A\in\Msym_{n}(\FF_2):\cork A=k\}|}{|\Msym_{n}(\FF_2)|}.\]
\end{theorem}
\begin{proof}
By Erd\H os--Kac theorem \cite[Proposition~3]{SieveErdosKac}, it suffices to show that for any $r$ satisfying \eqref{eq:rrange} we have
\[\left|\frac{|\{x\in S_r(N):\Delta_{\Q(\sqrt{x})}\in\mathcal{D} ,\ \rk_4 x=k\}|}{|\{x\in S_r(N):\Delta_{\Q(\sqrt{x})}\in\mathcal{D} \}|}-\lim_{n\rightarrow\infty} P(k|n)\right|\ll (\log\log N)^{-c}.\]
Fixing some $r$, we consider the distribution of the $4$-rank of the class groups of quadratic fields with discriminants in
\[
\{\Delta_{\Q(\sqrt{x})}\in\mathcal{D}: x\in S_{r}(N)\}.
\]
We can find some $W\subseteq S_r(N)$ that is comfortably spaced above $D_1$ and $(c_{10}\log\log\log N)$-regular by Theorem~\ref{theorem:small} and Siegel-less above $D_1$ by Proposition~\ref{6.10}, so that 
\[
|W|\geq (1-\epsilon)\Phi_r(N)\text{ with } \epsilon\ll (\log D_1)^{-1}+(\log N)^{-1/2}\ll (\log\log N)^{-c_9}. 
\]
Then applying Theorem~\ref{6.9}, we see that we can restrict to acceptable boxes by introducing an error $\ll (\log\log N)^{-c_9}$. In other words, it suffices to show that for any acceptable box $X$, we have 
\[
\left|\frac{|\{x\in X:\Delta_{\Q(\sqrt{x})}\in\mathcal{D} ,\ \rk_4 x=k\}|}{|\{x\in X:\Delta_{\Q(\sqrt{x})}\in\mathcal{D} \}|}-\lim_{n\rightarrow\infty} P(k|n)\right|\ll(\log\log N)^{-c}.
\]

Take $X$ to be an acceptable box, then one can check that there exists constants that satisfy the requirements in Theorem~\ref{6.4}, applying to prebox $(X_1\times\dots\times X_r, \varnothing)$.

Fix some $c_1>0$.
Take $k_0$ such that $t_{k_0}\leq D_1<t_{k_0+1}$,
$k_1$ such that $t_{k_1}\leq \exp(D_1^{c_1})<t_{k_1+1}$ and $k_2=r$. 
Take $2c_{10}<c_9<1$, $c_8>c_9$, $2c_7<1-c_9$.
We have
$\frac{1}{2}\log\log t_{k_1}\leq \frac{c_1}{2}\log D_1=\frac{c_1}{2}(\frac{1}{2}\log\log N)^{c_9}<r/3$, so $k_1<r/3$ and we have regular spacing at $k_1$.
Then
$k_1\leq 2\log\log t_{k_1}\leq 2c_1(\frac{1}{2}\log\log N)^{c_9}<r^{c_8}$.
Write $\eta:=c_{10}\log\log\log N$.
By regular spacing at $k_0$,
If $k_0\leq \eta$, we have
\[k_0<\frac{1}{2}\log\log t_{k_0}+\eta=\left(\frac{c_9}{2}+c_{10}\right)\log\log \log N<(1-2c_7)\log r.\]
If $k_0> \eta$, we have
\[k_0<\frac{1}{2}\log\log t_{k_0}+k_0^{4/5}\eta^{1/5}<(\left(\left(\frac{c_9}{2}\right)^{1/5}+\frac{1}{5}c_{10}^{1/5}\right)^5\log\log \log N<(1-2c_7)\log r.\]

For $k_0<i\leq r/3$, we have $(1-\frac{2c_{10}}{c_9})\frac{1}{2}\log\log t_i<i<(1+\frac{2c_{10}}{c_9})\frac{1}{2}\log\log t_i$ by regular spacing.
For $r/3<i\leq r$, we have
\begin{equation}\label{eq:rspace}
i<\left(1+\frac{2c_{10}}{c_9}\right)\frac{3}{2}\log\log t_{i/3}<\left(1+\frac{2c_{10}}{c_9}\right)\frac{3}{2}\log\log t_i
\end{equation}
by regular spacing at $i/3$.
Therefore $i<\log t_i$ for $k_0<i\leq r$.
 
We now pick $c_2,\dots,c_6$.
Take $c_1=c_2$ and $c_3>4+6c_{10}/c_9$.
By \eqref{eq:rspace}, 
\[\frac{\log (t_i-s_i)}{\log t_i}
>1-\frac{3\log \log t_i+\log \log D_1}{\log t_i}
>1-\frac{4\log \log D_1}{\log D_1},\]
so using lower bounds on the number of primes in short intervals (for example \cite{IwaniecJutila}) and \eqref{eq:rspace}, we have
$|X_i|\gg \frac{t_i-s_i}{\log t_i}$, so $|X_i|\geq e^{i}k_2^{c_7}t_i(\log t_i)^{-c_3}$ for all $k_0 < i\leq r$. 
Take $c_4>0$ then $r<D_1^{c_4}$.
By regular spacing at $k_1$, we have
\[j_{k_0+1}\geq k_0+c_5\log t_{k_0+1}>k_0+c_2^{-1}c_5\log\log t_{k_1}>k_1,
\]
when $c_5>c_2$ and large enough $N$.
Suppose $\log t_{j_i}<(\log t_i)^{c_6}$ for some $j_i>i-2+c_5\log t_i$,
then 
\[
j_i>i-2+c_5(\log t_{j_i})^{1/c_6}
>c_5(\log t_{k_1+1})^{1/c_6} 
>c_5(D_1)^{c_1/c_6}
>c_5\exp(c_1(\log\log N)^{c_{9}}/c_6)
\]
which is greater than $r$. 
Then Theorem~\ref{6.4} shows that the Legendre matrices of $x\in X$, are equidistributed amongst all $(r-1)\times(r-1)$ symmetric matrices over $\FF_2$ up to reordering some columns and rows, with an error within the statement.

To model the corank of the matrices of $x$, we begin with an empty matrix, then we add an extra column and row (keeping the matrix symmetric) in each step. We consider Markov chains on the non-negative integers with transition probabilities
\[p_{i,j}=
\begin{cases}
\hfil 2^{-i-1}&\text{if }j=i\text{ or }i+1,\\
1-2^{-i}&\text{if }j=i-1,\\
\hfil 0&\text{otherwise.}\\
\end{cases}\]
Here $p_{i,j}$ is the probability of obtaining from any matrix $A$ of corank $i$, a matrix of the form
\[\begin{pmatrix}
A &y\\
y^T&x
\end{pmatrix}\]
of corank $j$, when the column vector $y$ and $x\in \FF_2$ are chosen randomly over $\FF_2$ \cite[Lemma~4]{MacWilliams}.
The transition probabilities give a stationary distribution
\[\pi_j:=\frac{\alpha}{\prod_{i=1}^j (2^i-1)},\]
where $\alpha:=
\prod_{j\text{ odd}}(1-2^{-j})=\prod_{j=1}^{\infty}(1+2^{-j})^{-1}$.
Notice that $\pi_j=\lim_{n\rightarrow\infty}P(j|n)$.
We are interested in the distribution after $r-1$ steps of the Markov chain with starting state being the corank of the empty matrix. We want to measure how far this distribution is from the stationary.

Suppose $(X_s)_{s\geq 0}$, $(X'_s)_{s\geq 0}$, and $(Y_s)_{s\geq 0}$ are independent Markov chains with transition probabilities $(p_{i,j})$, starting at state $i$, state $j$, and the stationary distribution respectively. 
Take some constants $\frac{1}{4}+\frac{1}{\sqrt{2}}<A<C<1$ and $B=\sqrt{2}$.
Write 
\[
f_{i,j}(t):=\prob(\inf\{s\geq 0:X_s=X'_s=0\}=t).
\]
We claim that $f_{i,j}(t)\leq A^{t}B^{i+j}$ for any $i,j,t\geq 0$.
We have $f_{0,0}(0)=1$ and $f_{0,0}(t)=0$ for any $t>0$, so assume $i+j>0$.
Fix $i\leq j$ and carry out induction on $t$.
It takes at least $j$ steps for $(X_s')$ to reach $0$, so $f_{i,j}(t)>0$ only when $t\geq j$. The base case we have $t=j$ and 
\[f_{i,j}(j)< 1<(AB)^j\leq A^{j}B^{i+j}.\]
Suppose our claim holds for any state at $t-1$. Then
\[f_{i,j}(t)=\sum_{i'=i-1}^{i+1}
\sum_{j'=j-1}^{j+1}
p_{i,i'}p_{j,j'}f_{i',j'}(t-1)
\leq
A^{t-1}\sum_{i'=i-1}^{i+1}
p_{i,i'}B^{i'}\sum_{j'=j-1}^{j+1}
p_{j,j'}B^{j'}
<A^{t}B^{i+j},\]
since for any $n>0$ we have
\[p_{n,n+1}B
+p_{n,n}
+\frac{p_{n,n-1}}{B}
=\frac{1}{2^{n+1}}
+\frac{1}{\sqrt{2}}
\leq \frac{1}{4}+\frac{1}{\sqrt{2}}
<A.\]
This completes our claim.

Now let $T:=\inf\{s\geq 0:X_s=Y_s=0\}$, then
\[\prob(T=t)=
\sum_{j=0}^{\infty}\pi_j f_{i,j}(t)
\leq\alpha A^{t}B^i\sum_{j=0}^{\infty}\frac{B^{j}}{\prod_{k=1}^j (2^k-1)}
\ll A^{t}B^i.
\]
Take some $\epsilon>0$ such that $1+\epsilon<A^{-1}C$. We have 
\[
\EE[(1+\epsilon)^{T}]=\sum_{t=0}^{\infty}(1+\epsilon)^t\prob(T=t)
\ll B^i\sum_{t=0}^{\infty}\left((1+\epsilon)A\right)^t
= \frac{B^i}{1-(1+\epsilon)A}
<\frac{B^i}{1-C}
\ll B^i
.
\]
By Markov's inequality and the proof of \cite[Theorem~1.8.3]{Norris}, we have  
\[
\left|\prob(X_r=j)-\pi_j\right|\leq \prob(T\geq r)\leq \frac{\EE[(1+\epsilon)^{T}]}{(1+\epsilon)^r}
\ll \frac{B^i}{(1+\epsilon)^r}.
\]

Take $(X_s)_{s\geq 0}$ to be the Markov chain modelling the $4$-rank of the set
$\{x\in\mathcal{D}: x\in X\}$ which begins at state $i=0$. Take a constant $0<c<\log(1+\epsilon)$, we have
\[\left|\prob(X_{r-1}=j)-\pi_j\right|
=O(\exp(-c(r-1))),\]
which is within the error term.
\end{proof}

\section{Proof of main theorems}
\label{sProof}
Recall from the introduction that
\[
\mathcal{D}_{n, m}(X) = \{D\in \mathcal{D}(X): \rk_4\CL(D) = \rk_4\CL^+(D) = n\text{ and }\rk_8\CL^+(D) = m\}.
\]
We also define
\[
\mathcal{D}_n(X) = \{D\in \mathcal{D}(X): \rk_4\CL^+(D) = n\}.
\]
In this section we prove the following theorem. 

\begin{theorem}
\label{t7.1}
There are $A, N_0 > 0$ such that for all $N > N_0$ and all integers $n_2 \geq n_3 \geq 0$ we have
\[
\left|\left|\mathcal{D}_{n_2, n_3}(N)\right| - Q(n_2|n_3) \cdot \left|\mathcal{D}_{n_2}(N)\right|\right| \leq \frac{AN}{\log \log \log N},
\]
where $Q(n_2|n_3)$ is the probability that a uniformly chosen $(n_2 + 1) \times n_2$-matrix with coefficients in $\FF_2$ has rank $n_2 - n_3$ and bottom row consisting of only zeroes.
\end{theorem}

To prove this theorem, our first step is to reduce to sufficiently nice boxes $X$. We formalize this in our next definition.

\begin{mydef}
\label{dnice}
Let $r \geq 1$ be an integer, let $X = X_1 \times \ldots \times X_r$ be a box and let $N \geq 10^{10^{10}}$ be a real number. Put
\[
D_1 := e^{(\log \log N)^{1/10}}, \quad \eta := \sqrt{\log \log \log N}.
\]
We let $W$ be the maximal subset of $S_r(N)$ that is comfortably spaced above $D_1$, $\eta$-regular and disjoint from the sets $V_i$ in Proposition~\ref{6.10}. We call $X$ a nice box for $N$ if $X \subseteq S_r(N)$,
$X \cap W \neq \varnothing$ and
\begin{align}
\label{erbound}
\left|r - \frac{1}{2} \log \log N\right| \leq (\log \log N)^{2/3}.
\end{align}
\end{mydef}

\begin{prop}
\label{p7.3}
There are $A, N_0 > 0$ such that for all $N > N_0$, all nice boxes $X$ for $N$ and all integers $n_2 \geq n_3 \geq 0$ we have
\[
\left|\left|X \cap \mathcal{D}_{n_2, n_3}(N)\right| - Q(n_2|n_3) \cdot \left|X \cap \mathcal{D}_{n_2}(N)\right|\right| \leq \frac{A|X|}{\log \log \log N}.
\]
\end{prop}

\begin{proof}[Proof that Proposition~\ref{p7.3} implies Theorem~\ref{t7.1}] From Erd\H os--Kac \cite[Proposition~3]{SieveErdosKac} it follows that we only need to consider $r$ satisfying \eqref{erbound}. For each such $r$, we apply Proposition~\ref{6.9} with $W$ as in Definition~\ref{dnice}; the required lower bound for $|W|$ follows from the material in Section \ref{sPrimes} and Proposition~\ref{6.10}.
\end{proof}

Given a box $X$ and $a:\M \rightarrow \pm 1$, our next step is to reduce to $X(a)$. However, it turns out that we can not prove equidistribution for all $a:\M \rightarrow \pm 1$, but only if $a$ is \emph{generic} in the following sense.

\begin{mydef}
For a field $K$ and for integers $a, b \geq 0$, we denote by $\Mat(K, a, b)$ the set of $a \times b$-matrices with coefficients in $K$. Let $\iota$ be the unique group isomorphism between $\pm 1$ and $\FF_2$. We put
\[
\M := \{(i,j):1\leq i< j\leq r\}, \quad \mathcal{N} = \varnothing.
\]
Given $a:\M \rightarrow \pm 1$, we associate a matrix $A \in \Mat(\mathbb{F}_2, r, r)$ by setting for all $i < j$
\[
A(i, j) = \iota \circ a(i, j), \quad A(j, i) = \iota \circ a(i, j)
\]
and finally
\[
A(i, i) = \iota \circ \prod_{j = 1}^r a(i, j).
\]
Think of $\FF_2^r$ as column vectors. We define the vector space
\[
\mathcal{V}_{a,2} = \{v \in \FF_2^r: v^T A = 0\} = \{v \in \FF_2^r : Av = 0\}.
\]
Let $R := (1, \ldots, 1)$, so that $R \in \mathcal{V}_{a,2}$. Put $n_2(a) := -1 + \dim_{\FF_2} \mathcal{V}_{a,2}$. 

Let $N$ be a large real and let $X = X_1 \times \ldots \times X_r$ be a nice box for $N$. Choose an index $k_{\textup{gap}}$ such that the extravagant spacing of $X$ is between $k_{\textup{gap}}$ and $k_{\textup{gap}} + 1$. Set
\[
n_{\textup{max}} := \left \lfloor \sqrt{2\log \log \log \log N} \right \rfloor,
\]
We say that $a:\M \rightarrow \pm 1$ is generic for $X$ if $n_2(a) \leq n_{\textup{max}}$ and furthermore we have for all $S \in \mathcal{V}_{a,2} \setminus \langle R \rangle$ and all $i \in \FF_2$ that
\begin{align}
\label{eGenVar}
\left|\left|\left\{j \in [r] : \frac{k_{\textup{gap}}}{2} \leq j \leq k_{\textup{gap}} \textup{ and } \pi_j(S) = i\right\}\right| - \frac{k_{\textup{gap}}}{4}\right| \leq 2^{-10n_{\textup{max}}} \cdot r
\end{align}
and
\begin{align}
\label{eGenCheb}
\left|\left|\left\{j \in [r] : k_{\textup{gap}} < j \leq 2k_{\textup{gap}} \textup{ and } \pi_j(S) = i\right\}\right| - \frac{k_{\textup{gap}}}{2}\right| \leq 2^{-10n_{\textup{max}}} \cdot r.
\end{align}
\end{mydef}

We shall prove that the Artin pairing $\text{Art}_2$ is equidistributed in $X(a)$ under favorable circumstances. For this reason we make the following definition.

\begin{mydef}
We say that a bilinear pairing
\[
\Art_2 : \mathcal{V}_{a,2} \times \mathcal{V}_{a,2} \rightarrow \FF_2
\]
is valid if the right kernel contains $(1, \ldots, 1)$. Fix a basis $w_1, \ldots w_{n_2}, R$ for $\mathcal{V}_{a,2}$. Using this basis we may identify $\Art_2$ with a $(n_2 + 1) \times (n_2 + 1)$ matrix with coefficients in $\FF_2$. Since $(1, \ldots, 1)$ is in the right kernel, we may also naturally identify $\Art_2$ with a $(n_2 + 1) \times n_2$ matrix. Finally define for a box $X$
\[
X(a, \Art_2) := \{x \in X(a) : \textup{ the Artin pairing of } x \textup{ equals } \Art_2\}.
\]
\end{mydef}

If $X = X_1 \times \dots \times X_r$ is a box with $D_1$ sufficiently large, we recall that $k$ is the largest index such that $|X_k| = 1$.

\begin{prop}
\label{p7.4}
There are $A, N_0 > 0$ such that for all $N > N_0$, all nice boxes $X$ for $N$, all integers $n_2 \geq 0$, all generic $a:\M \rightarrow \pm 1$ for $X$ with $n_2(a) = n_2$ and
\begin{align}
\label{eXjlarge}
|X_j(a, x_1 \cup \dots \cup x_k)| \geq \frac{1}{(\log t_{k + 1})^{100}} \cdot |X_j|
\end{align}
for all $k < j \leq r$, and all valid Artin pairings $\Art_2$, we have
\[
\left||X(a, \Art_2)| - 2^{-n_2(n_2 + 1)}|X(a)|\right| \leq \frac{A|X(a)|}{(\log \log \log N)^3}.
\]
Here we write $x_1, \ldots, x_k$ for the unique elements of $X_1, \ldots, X_k$.
\end{prop}

\begin{proof}[Proof that Proposition~\ref{p7.4} implies Proposition~\ref{p7.3}]
Take $N$ to be a large integer and take $X$ to be a nice box for $N$. If $N$ is sufficiently large and $n_2 > n_{\text{max}}$, we have
\[
\lim_{k \rightarrow \infty} P(k|n_2) = O(\log \log \log N).
\]
Then it follows easily from the proof of Theorem~\ref{tFK} that
\[
\left|\left|X \cap \mathcal{D}_{n_2, n_3}(N)\right| - Q(n_2|n_3) \cdot \left|X \cap \mathcal{D}_{n_2}(N)\right|\right| \leq 2 \left|X \cap \mathcal{D}_{n_2}(N)\right| \leq \frac{A|X|}{\log \log \log N}
\]
for a sufficiently large constant $A > 0$. From now on suppose that $n_2 \leq n_{\text{max}}$. We deduce from Hoeffding's inequality that the proportion of $S$ in $\FF_2^r$ failing equation~\eqref{eGenVar} or equation~\eqref{eGenCheb} is bounded by
\[
O\left(\exp\left(-2^{-20n_{\text{max}}^2} \cdot k_{\text{gap}}\right)\right).
\]
Given $S \not \in \langle R \rangle$, the proportion of $a:\M \rightarrow \pm 1$ with $S \in \mathcal{V}_{a,2}$ is $O(0.5^r)$. Taking the union over all $S$ in $\FF_2^r$ failing equation~\eqref{eGenVar} or equation~\eqref{eGenCheb} proves that the proportion of non-generic $a$ is at most
\[
O\left(\exp\left(-2^{-20n_{\text{max}}^2} \cdot k_{\text{gap}}\right)\right).
\]
Put $k_2 := \lfloor0.25k_{\text{gap}}\rfloor$. Then we have for all $\sigma \in \mathcal{P}(k_2)$ that $a:\M \rightarrow \pm 1$ is generic if and only if $\sigma(a)$ is generic, where $\sigma(a)$ is defined in the natural way. Theorem~\ref{6.4} implies that
\[
\sum_{a:\M \rightarrow \pm 1} \left|2^{-|\M|} \cdot |X| - \frac{1}{k_2!} \sum_{\sigma \in \mathcal{P}(k_2)} |X(\sigma(a))|\right| \leq (k_2^{-\delta_1} + t_{k + 1}'^{-\delta_2}) \cdot |X|,
\]
where $\delta_1$ and $\delta_2$ are small, positive absolute constants. Restricting this sum to the non-generic $a$ shows that the union of $X(a)$ over all non-generic $a$ is within the error term of Proposition~\ref{p7.3}. We now deal with the $a:\M \rightarrow \pm 1$ that fail equation~\eqref{eXjlarge}. Let $j$ be an integer satisfying $k < j \leq r$. We say that $a, a': \M \rightarrow \pm 1$ are equivalent at $j$, which we write as $a \sim_j a'$, if $a(i, j) = a'(i, j)$ for all $1 \leq i \leq k$. Since our box is $\eta$-regular, we see that $k$ is roughly equal to $\log \log D_1$. In particular if $N$ is sufficiently large, we get
\[
k \leq 2 \log \log D_1 = \frac{1}{5} \log \log \log N.
\]
Then there are at most $2^{\frac{1}{5} \log \log \log N}$ equivalence classes. Furthermore, if $a:\M \rightarrow \pm 1$ is such that equation~\eqref{eXjlarge} fails for some fixed $j$, we have that
\[
\left|\bigcup_{a': a \sim_j a'} X(a')\right| \leq \frac{1}{(\log t_{k + 1})^{100}} \cdot |X|,
\]
where the union is over all $a':\M \rightarrow \pm 1$ equivalent to $a:\M \rightarrow \pm 1$ at $j$. Summing this over all choices of $j$ and all equivalence classes, we stay within the error term of Proposition~\ref{p7.3}. So far we have shown
\begin{multline*}
\left|\left|X \cap \mathcal{D}_{n_2, n_3}(N)\right| - Q(n_2|n_3) \cdot \left|X \cap \mathcal{D}_{n_3}(N)\right|\right| \\
\leq 
\sum_{\substack{a \text{ generic} \\ a \text{ sat. eq. } \eqref{eXjlarge} \\ n_2(a) = n_2}} \sum_{\substack{\text{rk}(\text{Art}_2) = n_2 - n_3 \\ \text{Art}_2 \text{ valid}}} \left||X(a, \Art_2)| - 2^{-n_2(n_2 + 1)}|X(a)|\right| + \frac{A|X|}{\log \log \log N}.
\end{multline*}
Note that we could have further restricted the sum over Artin pairings to only those with bottom row identically $0$. However, the displayed inequality suffices for our purposes. We now apply Proposition~\ref{p7.4} for every generic $a:\M \rightarrow \pm 1$ for $X$ such that it satisfies equation~\eqref{eXjlarge} and $n_2(a) = n_2$, and all valid Artin pairings $\text{Art}_2$ with $\text{rk}(\text{Art}_2) = n_2 - n_3$. Since there are at most
\[
2^{n_2(n_2 + 1)} \leq 2^{n_{\text{max}}(n_{\text{max}} + 1)}
\]
valid Artin pairings, we get
\[
\sum_{\substack{a \text{ generic} \\ a \text{ sat. eq. } \eqref{eXjlarge} \\ n_2(a) = n_2}} \hspace{-0.1cm} \sum_{\substack{\text{rk}(\text{Art}_2) = n_2 - n_3 \\ \text{Art}_2 \text{ valid}}} \hspace{-0.1cm} \left||X(a, \Art_2)| - 2^{-n_2(n_2 + 1)}|X(a)|\right| \leq 2^{n_{\text{max}}(n_{\text{max}} + 1)} \frac{A|X|}{(\log \log \log N)^3}
\]
as desired.
\end{proof}

\begin{mydef}
Let $X$ be a box and $Y \subseteq X$ a subset. Let $S \subseteq [r]$ and let $Q \in \prod_{i \in S} X_i$. We define
\[
Y(Q) := \{y \in Y : \pi_{S}(y) = Q\}.
\]
We shall slightly abuse notation by writing $X(a, Q)$ for $X(a)(Q)$. If $i \not \in S$, we also define for a subset $Z \subseteq \prod_{i \in S} X_i$
\[
X_i(a, Z) := \left\{x \in X_i : \textup{ for all } j \in S, Q \in Z \textup{ we have } \left(\frac{x}{\pi_j(Q)}\right) = a(i, j)\right\}.
\]
Note that this is a natural generalization of $X_j(a, Q)$ as defined in Definition~\ref{dMa}.
\end{mydef}

In our next definition we introduce variable indices, which are by definition certain subsets $S$ of $[r]$. At the very end of this section we will reduce to the case where we have chosen one element $x_i \in X_i$ for all $i \in [r] - S$, whence the terminology.

\begin{mydef}
\label{dVarInd}
Let $a:\M \rightarrow \pm 1$. Recall that we fixed a basis $w_1, \ldots w_{n_2}, R$ for $\mathcal{V}_{a,2}$. Let $1 \leq j_1 \leq n_2 + 1$ and let $1 \leq j_2 \leq n_2$. 
Let $E_{j_1, j_2}$ be the $(n_2 + 1) \times n_2$-matrix with $E_{j_1, j_2}(j_1, j_2) = 1$ and $0$ otherwise, and let $F_{j_1, j_2}$ be the dual basis. Any non-zero multiplicative character $F: \Mat(\FF_2, n_2 + 1, n_2) \rightarrow \pm 1$ can be written as
\[
F = \iota^{-1} \circ \sum_{\substack{1 \leq j_1 \leq n_2 + 1 \\ 1 \leq j_2 \leq n_2}} c_{j_1, j_2} F_{j_1, j_2}
\]
with not all $c_{j_1, j_2}$ zero. A set $S \subseteq [r]$ is called a set of variable indices for $F$ if there are $i_1(F), i_2(F) \in S$ such that
\[
\frac{k_{\textup{gap}}}{2} \leq i \leq k_{\textup{gap}} \textup{ for all } i \in S \setminus \{i_2(F)\}, \quad k_{\textup{gap}} < i_2(F) \leq 2k_{\textup{gap}}
\]
and
\begin{itemize}
\item if $c_{n_2 + 1, j_2} = 0$ for all $1 \leq j_2 \leq n_2$ and $c_{j_1, j_1} = 0$ for all $1 \leq j_1 \leq n_2$ and $c_{j_1, j_2} = 0$ implies $c_{j_2, j_1} = 0$ for all $1 \leq j_1, j_2 \leq n_2$, we choose any pair $(j_1, j_2)$ such that $c_{j_1, j_2} = 1$. Furthermore, choose $|S(F)| = 2$,
\[
i_1(F) \in \bigcap_{i \neq j_1} \{j \in [r] : \pi_j(w_i) = 0\} \cap \{j \in [r] : \pi_j(w_{j_1}) = 1\}
\]
and
\[
i_2(F) \in \bigcap_{i \neq j_2} \{j \in [r] : \pi_j(w_i) = 0\} \cap \{j \in [r] : \pi_j(w_{j_2}) = 1\};
\]
\item if there are $1 \leq j_1, j_2 \leq n_2$ such that $c_{j_1, j_2} = 1$ and $c_{j_2, j_1} = 0$, choose such a pair $(j_1, j_2)$. Next choose $|S(F)| = 3$ and
\[
S(F) \subseteq \bigcap_{i \not \in \{j_1, j_2\}} \{j \in [r] : \pi_j(w_i) = 0\}
\]
and
\[
S(F) \cap \{j \in [r] : \pi_j(w_{j_1}) = 1, \pi_j(w_{j_2}) = 0\} = \{i_1(F)\}
\]
and
\[
S(F) \cap \{j \in [r] : \pi_j(w_{j_2}) = 1, \pi_j(w_{j_1}) = 0\} = \{i_2(F)\}
\]
and
\[
S(F) \cap \{j \in [r] : \pi_j(w_{j_1}) = 1, \pi_j(w_{j_2}) = 1\} = \varnothing;
\]
\item in all other cases, choose a pair $(j_2, j_2)$ such that $c_{j_2, j_2} = 1$ or choose a pair $(n_2 + 1, j_2)$ such that $c_{n_2 + 1, j_2} = 1$. We pick $|S(F)| = 2$ and
\[
i_1(F) \in \bigcap_{i  \neq j_2} \{j \in [r] : \pi_j(w_i) = 0\} \cap \{j \in [r] : \pi_j(w_{j_2}) = 1\}
\]
and
\[
i_2(F) \in \bigcap_{i = 1}^{n_2} \{j \in [r] : \pi_j(w_i) = 0\}.
\]
\end{itemize}
\end{mydef}

If $a:\M \rightarrow \pm 1$ is generic for $X$, we will now show that one can find variable indices provided that $r$ is sufficiently large. Our essential tool is the following combinatorial lemma.

\begin{lemma}
Assume that $a:\M \rightarrow \pm 1$ is generic for $X$. If $w_1, \ldots, w_d, R \in \mathcal{V}_{a,2}$ are linearly independent, then we have for all $\mathbf{v} \in \FF_2^d$
\[
\left|\left|\left\{i \in [r] : \frac{k_{\textup{gap}}}{2} \leq i \leq k_{\textup{gap}} \textup{ and } \pi_i(w_j) = \pi_j(\mathbf{v}) \textup{ for all } 1 \leq j \leq d\right\}\right| - \frac{k_{\textup{gap}}}{2^{d + 1}}\right| \leq 3^d \cdot 2^{-10n_{\textup{max}}} \cdot r.
\]
\end{lemma}

\begin{proof}
We proceed by induction on $d$. The base case $d = 1$ follows immediately from equation~\eqref{eGenVar}. Now suppose that $d > 1$. We define for $\mathbf{w} \in \FF_2^d$
\[
g(\mathbf{w}) = \left|\left\{i \in [r] : \frac{k_{\textup{gap}}}{2} \leq i \leq k_{\textup{gap}} \textup{ and } \pi_i(w_j) = \pi_j(\mathbf{w}) \textup{ for all } 1 \leq j \leq d\right\}\right|.
\]
Let $\mathbf{v} \in \FF_2^d$ be given. Let $\mathbf{v}_1, \mathbf{v}_2, \mathbf{v}_3$ be the three unique pairwise distinct vectors such that $\pi_{d - 2}(\mathbf{v}_i) = \pi_{d - 2}(\mathbf{v})$ and $\mathbf{v}_i \neq \mathbf{v}$. We have
\begin{align*}
2\left|g(\mathbf{v}) - \frac{k_{\textup{gap}}}{2^{d + 1}}\right| &\leq \left|3g(\mathbf{v}) + \sum_{i = 1}^3 g(\mathbf{v}_i) - \frac{3k_{\textup{gap}}}{2^d}\right| + \left|\frac{k_{\textup{gap}}}{2^{d - 1}} - g(\mathbf{v}) - \sum_{i = 1}^3 g(\mathbf{v}_i)\right| \\
&\leq \sum_{i = 1}^3 \left|g(\mathbf{v}) + g(\mathbf{v}_i) - \frac{k_{\textup{gap}}}{2^d}\right| + \left|\frac{k_{\textup{gap}}}{2^{d - 1}} - g(\mathbf{v}) - \sum_{i = 1}^3 g(\mathbf{v}_i)\right|.
\end{align*}
Now apply the induction hypothesis.
\end{proof}

With this lemma it is straightforward to find variable indices provided that $a$ is generic for $X$ and $r$ is sufficiently large. We can now formulate our next reduction step. For a subset $T \subseteq [r]$, a point $P \in \prod_{i \in T} X_i$ and $a:\M \rightarrow \pm 1$, we say that $P$ is consistent with $a$ if
\[
\left(\frac{\pi_i(P)}{\pi_j(P)}\right) = a(i, j)
\]
for all distinct $i, j \in T$ with $i < j$.

\begin{prop}
\label{p7.4b}
There are $A, N_0 > 0$ such that for all $N > N_0$, all nice boxes $X$ for $N$, all integers $n_2 \geq 0$, all generic $a:\M \rightarrow \pm 1$ for $X$ with $n_2(a) = n_2$, all non-zero multiplicative characters $F$ from $\Mat(\FF_2, n_2 + 1, n_2)$ to $\FF_2$, all sets of variable indices $S$ for $F$ and all $Q \in [k_{\textup{gap}}] - S$ consistent with $a$ such that
\begin{align}
\label{eXjaPlarge}
|X_j(a, Q)| \geq 4^{-k_{\textup{gap}}} \cdot |X_j|
\end{align}
for all $j \in S$, we have
\[
\left|\sum_{x \in X(a, Q)} F(\Art_2(x))\right| \leq \frac{A|X(a, Q)|}{(\log \log \log N)^3}.
\]
\end{prop}

\begin{proof}[Proof that Proposition~\ref{p7.4b} implies Proposition~\ref{p7.4}]
Let $F$ be a non-zero multiplicative character from $\Mat(\FF_2, n_2 + 1, n_2)$ to $\FF_2$. We claim that there exist absolute constants $A', N_0' > 0$ such that for all $N > N_0'$
\begin{align}
\label{eInt}
\left|\sum_{x \in X(a)} F(\Art_2(x))\right| \leq \frac{A'|X(a)|}{(\log \log \log N)^{3}}.
\end{align}
Once we establish equation~\eqref{eInt}, Proposition~\ref{p7.4} follows easily. There exist $j_1$ and $j_2$ such that $F$ depends minimally on $(j_1, j_2)$. Take a set of variable indices $S$ for $(j_1, j_2)$. We split the sum in equation~\eqref{eInt} over all $Q \in [k_{\text{gap}}] - S$ consistent with $a$. If $Q$ satisfies equation~\eqref{eXjaPlarge} for all $j \in S$, we apply Proposition~\ref{p7.4b} with this $(j_1, j_2)$ and $S$. It remains to bound
\begin{align}
\label{eBadP}
\sum_{\substack{Q \in [k_{\text{gap}}] - S \\ Q \text{ consistent with } a \\ Q \text{ fails eq. } \eqref{eXjaPlarge}}} |X(a, Q)|.
\end{align}
But this follows quickly from an application of Proposition~\ref{pPointSub} with the prebox
\[
(X_{k + 1}(a, P) \times \dots \times X_r(a, P), P),
\]
where $P$ is the union of $x_1, \dots, x_k$. Note that we make crucial usage of equation~\eqref{eXjlarge} to validate the fourth condition of Proposition~\ref{pPointSub}.
\end{proof}

It remains to prove Proposition~\ref{p7.4b}, which we shall do now.

\begin{proof}[Proof of Proposition~\ref{p7.4b}] Put
\[
M := \left\lfloor (\log \log \log N)^{10} \right\rfloor, \quad S' := [k_{\text{gap}}] \cap S, \quad m := |S'|.
\]
Define
\[
X' := \prod_{i \in S'} X_i(a, Q),
\]
and
\[
Y := \left\{x \in X' : x \text{ is consistent with } a\right\}.
\]
Also set $R := \lfloor\exp\left(\exp\left(0.2k_{\text{gap}}\right)\right)\rfloor$. We let $Z_{\text{var}}^1, \ldots, Z_{\text{var}}^t$ be a longest sequence of subsets of $X'$ satisfying
\begin{itemize}
\item we have for all $1 \leq s \leq t$ the equality
\[
Z_{\text{var}}^s = \prod_{i \in S'} Z_i^s
\]
for some subset $Z_i^s$ of $X_i(a, Q)$ with cardinality $M$;
\item we have $Z_{\text{var}}^s \subseteq Y$ and every $y \in Y$ is in at most $R$ different $Z_{\text{var}}^s$;
\item for all distinct $1 \leq s, s' \leq t$ we have $\left|Z_{\text{var}}^s \cap Z_{\text{var}}^{s'}\right| \leq 1$.
\end{itemize}
Define $Y_{\text{bad}}$ as
\[
Y_{\text{bad}} := \left\{y \in Y : \left|\left\{1 \leq s \leq t : y \in Y_{\text{bad}} \right\}\right| < R\right\}
\]
and let $\delta$ be the density of $Y_{\text{bad}}$ in $X'$. With a greedy algorithm, we can construct a subset $W$ of $Y_{\text{bad}}$ of density at least $\delta/RM^m$ such that $|W \cap Z_{\text{var}}^s| \leq 1$ for all $s$. If there were to be subsets $Z_i \subseteq X_i(a, Q)$ for each $i \in S'$ satisfying $|Z_i| = M$ and
\[
\prod_{i \in S'} Z_i \subseteq W,
\]
we could extend our sequence $Z_{\text{var}}^1, \ldots, Z_{\text{var}}^t$ to a longer sequence. Hence we may apply the contrapositive of Proposition~4.1 of Smith \cite{Smith} to infer
\[
M > \frac{\exp(0.3k_{\text{gap}})}{5 \log(RM^m/\delta)},
\]
since $|X_i(a, Q)| \geq \exp(\exp(0.3k_{\text{gap}}))$ for sufficiently large $N$ thanks to equation~\eqref{eXjaPlarge} and the regular spacing. This yields
\begin{align}
\label{eUpperd}
\delta < \frac{RM^m}{\exp\left(\frac{\exp(0.3k_{\text{gap}})}{5M}\right)} \leq \exp(-0.25\exp(k_{\text{gap}}))
\end{align}
if $N$ is sufficiently large. A straightforward application of the Chebotarev Density Theorem, see Theorem~\ref{cheb}, shows that for $i > k_{\text{gap}}$
\begin{align}
\label{eCheb1}
|X_i(a, Q \cup Z)| = \frac{|X_i(a, Q)|}{2^{(M - 1)^m}} \left(1 + O\left(e^{-2k_{\text{gap}}}\right)\right),
\end{align}
where we made use of the extravagant spacing of $k_{\text{gap}}$. Then Proposition~\ref{6.3} implies that for each $y \in Y$ the quantity $X(a, Q \times \{y\})$ is of the expected size. Hence equation~\eqref{eUpperd} implies that
\[
\left|\sum_{\substack{x \in X(a, Q) \\ \pi_{S'}(x) \in Y_{\text{bad}}}} F(\Art_2(x))\right| \leq \sum_{\substack{x \in X(a, Q) \\ \pi_{S'}(x) \in Y_{\text{bad}}}} 1
\]
is easily within the error of our proposition. Given $Z_{\text{var}}^s$, we define
\[
\text{Hull}(Z_{\text{var}}^s) := \{Q\} \times Z_{\text{var}}^s \times \prod_{j \in [r] - [k_{\text{gap}}]} X_j(a, Q \cup Z_{\text{var}}^s).
\]
For each $x \in X(a, Q)$ with $\pi_{S'}(x) \not \in Y_{\text{bad}}$ we define the counting function
\[
\Lambda(x) := \left|\left\{1 \leq s \leq t: x \in \text{Hull}(Z_{\text{var}}^s)\right\}\right|.
\]
We shall compute the first and second moment of $\Lambda(x)$. Since the second moment will turn out to be approximately the square of the first moment, we see that the value of $\Lambda(x)$ is roughly constant. Then we shall use this to reduce to spaces of the shape $\text{Hull}(Z_{\text{var}}^s) \cap X(a, Q)$.

We start by computing the first moment as follows
\begin{align*}
\sum_{\substack{x \in X(a, Q) \\ \pi_{S'}(x) \not \in Y_{\text{bad}}}} \Lambda(x) 
&= \sum_{y \in Y \setminus Y_{\text{bad}}} \sum_{\substack{x \in X(a, Q) \\ \pi_{S'}(x) = y}} \sum_{1 \leq s \leq t} \mathbf{1}_{x \in \text{Hull}(Z_{\text{var}}^s)} \\
&= \sum_{y \in Y \setminus Y_{\text{bad}}} \sum_{1 \leq s \leq t} \left|X(a, Q) \cap \text{Hull}(Z_{\text{var}}^s) \cap \pi_{S'}^{-1}(y)\right|.
\end{align*}
The last expression is obviously $0$ if $y \not \in Z_{\text{var}}^s$. If $y \in Z_{\text{var}}^s$, we make an appeal to equation~\eqref{eCheb1} and Proposition~\ref{6.3} to deduce
\[
\left|X(a, Q) \cap \text{Hull}(Z_{\text{var}}^s) \cap \pi_{S'}^{-1}(y)\right| = \frac{\left|X(a, Q) \cap \pi_{S'}^{-1}(y)\right|}{2^{(M - 1)^m \cdot |[r] - [k_{\text{gap}}]|}}\left(1 + O\left(e^{-k_{\text{gap}}}\right)\right).
\]
Since there are precisely $R$ values of $s$ such that $y \in Z_{\text{var}}^s$, we conclude that the first moment of $\Lambda(x)$ is equal to
\[
\frac{R\left|X(a, Q) \cap \pi_{S'}^{-1}(y)\right|}{2^{(M - 1)^m \cdot |[r] - [k_{\text{gap}}]|}}\left(1 + O\left(e^{-k_{\text{gap}}}\right)\right).
\]
To compute the second moment, we expand $\Lambda(x)^2$ as
\[
\sum_{\substack{x \in X(a, Q) \\ \pi_{S'}(x) \not \in Y_{\text{bad}}}} \Lambda(x)^2
= \sum_{y \in Y \setminus Y_{\text{bad}}} \sum_{\substack{x \in X(a, Q) \\ \pi_{S'}(x) = y}} \sum_{1 \leq s \leq t} \sum_{1 \leq s' \leq t} \mathbf{1}_{x \in \text{Hull}(Z_{\text{var}}^s)} \mathbf{1}_{x \in \text{Hull}(Z_{\text{var}}^{s'})},
\]
which we split as
\[
\sum_{y \in Y \setminus Y_{\text{bad}}} \sum_{\substack{x \in X(a, Q) \\ \pi_{S'}(x) = y}} \sum_{1 \leq s \leq t} \mathbf{1}_{x \in \text{Hull}(Z_{\text{var}}^s)} + \sum_{y \in Y \setminus Y_{\text{bad}}} \sum_{\substack{x \in X(a, Q) \\ \pi_{S'}(x) = y}} \sum_{\substack{1 \leq s, s' \leq t \\ s \neq s'}} \mathbf{1}_{x \in \text{Hull}(Z_{\text{var}}^s) \cap \text{Hull}(Z_{\text{var}}^{s'})}.
\]
We have already seen how to deal with the first sum. To treat the second sum, we first rewrite it as
\[
\sum_{y \in Y \setminus Y_{\text{bad}}} \sum_{\substack{1 \leq s, s' \leq t \\ s \neq s'}} \left|X(a, Q) \cap \text{Hull}(Z_{\text{var}}^s) \cap \text{Hull}(Z_{\text{var}}^{s'}) \cap \pi_{S'}^{-1}(y)\right|.
\]
Next observe that the above sum is zero if $y \not \in Z_{\text{var}}^s \cap Z_{\text{var}}^{s'}$. If $y \in Z_{\text{var}}^s \cap Z_{\text{var}}^{s'}$, we have, again due to the Chebotarev Density Theorem and Proposition~\ref{6.3}, that
\[
\left|X(a, Q) \cap \text{Hull}(Z_{\text{var}}^s) \cap \text{Hull}(Z_{\text{var}}^{s'}) \cap \pi_{S'}^{-1}(y)\right| = \frac{\left|X(a, Q) \cap \pi_{S'}^{-1}(y)\right|}{2^{2(M - 1)^m \cdot |[r] - [k_{\text{gap}}]|}}\left(1 + O\left(e^{-k_{\text{gap}}}\right)\right).
\]
There are precisely $R^2 - R$ pairs of $(s, s')$ such that $y \in Z_{\text{var}}^s \cap Z_{\text{var}}^{s'}$ and $s \neq s'$. Hence the second moment equals
\begin{multline*}
\left(\frac{(R^2 - R)\left|X(a, Q) \cap \pi_{S'}^{-1}(y)\right|}{2^{2(M - 1)^m \cdot |[r] - [k_{\text{gap}}]|}} + \frac{R\left|X(a, Q) \cap \pi_{S'}^{-1}(y)\right|}{2^{(M - 1)^m \cdot |[r] - [k_{\text{gap}}]|}}\right)\left(1 + O\left(e^{-k_{\text{gap}}}\right)\right) \\
=\frac{R^2\left|X(a, Q) \cap \pi_{S'}^{-1}(y)\right|}{2^{2(M - 1)^m \cdot |[r] - [k_{\text{gap}}]|}}\left(1 + O\left(e^{-k_{\text{gap}}}\right)\right).
\end{multline*}
Having computed the first and second moment, we apply Chebyshev's inequality to deduce that outside a set of density $O\left(e^{-0.5k_{\text{gap}}}\right)$ in the subset of those $x \in X(a, Q)$ satisfying $\pi_{S'}(x) \not \in Y_{\text{bad}}$, we have that
\[
\left|\Lambda(x) - \frac{R\left|X(a, Q) \cap \pi_{S'}^{-1}(y)\right|}{2^{(M - 1)^m \cdot |[r] - [k_{\text{gap}}]|}}\right| \leq \frac{R\left|X(a, Q) \cap \pi_{S'}^{-1}(y)\right|}{2^{(M - 1)^m \cdot |[r] - [k_{\text{gap}}]|}}e^{-0.25k_{\text{gap}}}.
\]
From this, we easily deduce that it suffices to prove that
\[
\left|\sum_{x \in X(a, Q) \cap \text{Hull}(Z_{\text{var}}^s)} F(\Art_2(x))\right| \leq \frac{A|X(a, Q) \cap \text{Hull}(Z_{\text{var}}^s)|}{(\log \log \log N)^3}.
\]
Since we are only dealing with one $Z_{\text{var}}^s$ at the time, we will abbreviate it as $Z$. If $m = 2$, we will also write $Z = Z_1 \times Z_2$ with $i_1(F) \in Z_1$.

We will now define a field $L$ depending on the shape of $F$ as in Definition~\ref{dVarInd}. If we are in the first case, we have $m = 1$ and we set
\[
L := \prod_{(p_1, p_2) \in Z \times Z} \phi_{p_1p_2, -1}.
\]
Here $\phi_{a, b}$ denotes any (fixed) choice of element in $\mathcal{F}_{a, b}^{\text{unr}}$. If we are instead in the second case, we have $m = 2$ and we define
\[
L := \prod_{(p_1, p_2, q_1, q_2) \in Z_1 \times Z_1 \times Z_2 \times Z_2} \phi_{p_1p_2, q_1q_2}.
\]
Finally, if we are in the first case, we have $m = 1$ again and we put
\[
L := \prod_{(p_1, p_2) \in Z \times Z} \phi_{p_1p_2, x}.
\]
where
\[
x := (p_1p_2)^{c_{j_2, j_2}} \cdot (-1)^{c_{n_2 + 1, j_2}}.
\]
Let $K$ be the largest multiquadratic extension of $\Q$ inside $L$. In each case we have a natural isomorphism
\begin{align}
\label{eGaladd}
\Gal(L/K) \cong \mathcal{A}(Z).
\end{align}
In the first case, this isomorphism is given by
\[
\sigma \mapsto \left((p_1, p_2) \mapsto \Frob_{\phi_{p_1p_2, -1}/\Q}(\sigma)\right),
\]
while in the second and third case it is respectively given by
\[
\sigma \mapsto \left((p_1, p_2, q_1, q_2) \mapsto \Frob_{\phi_{p_1p_2, q_1q_2}/\Q}(\sigma)\right), \quad \sigma \mapsto \left((p_1, p_2) \mapsto \Frob_{\phi_{p_1p_2, x}/\Q}(\sigma)\right).
\]
Note that any prime $p \in X_j(a, Q)$ splits completely in $K$ by construction. Given $\sigma \in \Gal(L/K)$, we define $X_j(a, Q \cup Z, \sigma)$ be the subset of primes $p \in X_j(a, Q \cup Z)$ that map to $\sigma$ under Frobenius. Then Lemma~\ref{lIm} and Proposition~\ref{cheb} yield
\[
|X_{i_2(F)}(a, Q \cup Z, \sigma)| = \frac{|X_{i_2(F)}(a, Q \cup Z)|}{2^{(M - 1)^m}}\left(1 + O\left(e^{-k_{\text{gap}}}\right)\right).
\]
Proposition~\ref{pPointSub} shows that for almost all choices of $Q_{\text{gap}} \in \prod_{[r] - [k_{\text{gap}}] - i_2(F)} X_j(a, Q \cup Z)$ consistent with $a$, we have that $|X_{i_2(F)}(a, Q \cup Q_{\text{gap}} \cup Z)|$ is of the expected size and furthermore
\begin{align}
\label{eChebAdd}
|X_{i_2(F)}(a, Q \cup Q_{\text{gap}} \cup Z, \sigma)| = \frac{|X_{i_2(F)}(a, Q \cup Q_{\text{gap}} \cup Z)|}{2^{(M - 1)^m}}\left(1 + O\left(e^{-k_{\text{gap}}}\right)\right)
\end{align}
for all $\sigma \in \Gal(L/K)$. By construction we have that
\[
Z_{\text{final}} := \{Q\} \times Z \times \{Q_{\text{gap}}\} \times X_{i_2(F)}(a, Q \cup Q_{\text{gap}} \cup Z) \subseteq X(a, Q),
\]
so it suffices to prove that
\begin{align}
\label{eFinal}
\left|\sum_{x \in Z_{\text{final}}} F(\Art_2(x))\right| \leq \frac{A|Z_{\text{final}}|}{(\log \log \log N)^3}.
\end{align}
Now pick
\[
\epsilon := \frac{1}{(\log \log \log N)^3}.
\]
We formally apply Theorem~\ref{tBadg} to $Z \times [M]$. We see that Theorem~\ref{tBadg} guarantees the existence of $g_{\text{spec}} \in \mathcal{A}(Z \times [M])$ such that $g_{\text{spec}}$ is not $\epsilon$-bad. Now pick any $x_1, \ldots, x_M \in X_{i_2(F)}(a, Q \cup Q_{\text{gap}} \cup Z)$. Then we can define a map $[M] \times [M] \rightarrow \Gal(L/K)$ by
\[
g(i, j) := \Frob_{L/K}(x_i) + \Frob_{L/K}(x_j),
\]
which we can naturally view as a map $[M] \times [M] \rightarrow \mathcal{A}(Z)$ due to the isomorphism in equation~\ref{eGaladd}. Hence $g$ naturally becomes an element of $\mathcal{A}(Z \times [M])$. 

We claim that we can find disjoint ordered subsets $A_1, \ldots, A_k$ of $X_{i_2(F)}(a, Q \cup Q_{\text{gap}} \cup Z)$ whose union is the whole set $X_{i_2(F)}(a, Q \cup Q_{\text{gap}} \cup Z)$ except for a small remainder such that defining $g$ as above for each $A_1, \ldots, A_k$, we get $g_{\text{spec}}$ under the natural identifications.

Let $g'_{\text{spec}} : [M] \times [M] \rightarrow \Gal(L/K)$ be the map that is sent to $g_{\text{spec}}$ under the natural identifications. Suppose that elements $x_1, \ldots, x_M \in X_{i_2(F)}(a, Q \cup Q_{\text{gap}} \cup Z)$ be given. Now look at the equation
\[
g'_{\text{spec}}(i, j) := \Frob_{L/K}(x_i) + \Frob_{L/K}(x_j).
\]
We see that one can freely choose $x_1$, and then all the $\Frob_{L/K}(x_j)$ for $j > 1$ are uniquely determined by $g'_{\text{spec}}(i, j)$ and $\Frob_{L/K}(x_1)$. Now an appeal to equation~\eqref{eChebAdd} finishes the proof of our claim.

Now pick one of the $A_i$ and suppose that $A_i = \{x_1, \ldots, x_M\}$. Let $\widetilde{F}: Z_{\text{final}} \rightarrow \FF_2$ be the map that sends $x$ to $\iota \circ F(\text{Art}_2(x))$. We can restrict $\widetilde{F}$ to $A_i$ and then naturally view $\widetilde{F}$ as a map from $Z \times [M]$ to $\FF_2$. Theorem~\ref{tBadg} then implies equation~\eqref{eFinal} and therefore Proposition~\ref{p7.4b} provided that we can verify the identity $d\widetilde{F} = g'_{\text{spec}}$.

We distinguish three cases depending on the type of $F$ as in Definition~\ref{dVarInd}. In the first case, we apply Theorem~\ref{t.2.8swapmin} and Theorem~\ref{t.2.8swapped}. Let $(j_1, j_2)$ be the entry as chosen in Definition~\ref{dVarInd}, so that $c_{j_1, j_2} = c_{j_2, j_1} = 1$. Theorem~\ref{t.2.8swapped} gives
\[
d\widetilde{F}_{j_1, j_2} = g'_{\text{spec}},
\]
where $\widetilde{F}_{j_1, j_2}$ is obtained from $F_{j_1, j_2}$ in the same way as $\widetilde{F}$ was obtained from $F$. Now consider any $(j_3, j_4)$, not equal to $(j_1, j_2)$, with $1 \leq j_3 \leq n_2 + 1$, $1 \leq j_4 \leq n_2$ and $c_{j_3, j_4} = 1$. Then we have $j_3 \leq n_2$ and $c_{j_4, j_3} = 1$. Hence Theorem~\ref{t.2.8swapmin} implies
\[
d\widetilde{F}_{j_3, j_4} = 0.
\]
Altogether we conclude that $d\widetilde{F} = g'_{\text{spec}}$.

We now deal with the second case. Once more let $(j_1, j_2)$ be the entry as chosen in Definition~\ref{dVarInd}, so that $1 \leq j_1, j_2 \leq n_2$, $c_{j_1, j_2} = 1$ and $c_{j_2, j_1} = 0$. Two applications of part (ii) of Theorem~\ref{t2.8} show that
\[
d\widetilde{F}_{j_1, j_2} = g'_{\text{spec}}.
\]
Two applications of Theorem~\ref{t2.8self} show that for all $1 \leq j_2 \leq n_2$
\[
d\widetilde{F}_{j_2, j_2} = 0,
\]
while two applications of part (i) of Theorem~\ref{t2.8} imply
\[
d\widetilde{F}_{j_3, j_4} = 0
\]
for all $1 \leq j_3 \leq n_2 + 1$, $1 \leq j_4 \leq n_2$ such that $(j_1, j_2) \not \in \{(j_3, j_4), (j_4, j_3)\}$ and $j_3 \neq j_4$. This finishes the proof of the second case.

It remains to treat the third case, which follows from an application of Theorem~\ref{t2.8} and Theorem~\ref{t2.8self}.
\end{proof}


\begin{thebibliography}{19}
\bibitem{Alladi}
K.~Alladi.
\newblock An {E}rd{\H o}s-{K}ac theorem for integers without large prime
  factors.
\newblock {\em Acta Arith.}, 49(1):81--105, 1987.

\bibitem{probabilistic}
N.~Alon and J.~H. Spencer.
\newblock {\em The probabilistic method}.
\newblock Wiley Series in Discrete Mathematics and Optimization. John Wiley \&
  Sons, Inc., Hoboken, NJ, fourth edition, 2016.

\bibitem{FK1}
\'{E}. Fouvry and J. Kl\"{u}ners.
\newblock On the $4$-rank of class groups of quadratic number fields.
\newblock {\em Invent. Math.}, 167(3):455--513, 2007.

\bibitem{FK2}
\'{E}. Fouvry and J. Kl\"{u}ners.
\newblock On the negative Pell equation.
\newblock {\em Ann. of Math. (2)}, 172(3):2035--2104, 2010.

\bibitem{FK3}
\'{E}. Fouvry and J. Kl\"{u}ners.
\newblock The parity of the period of the continued fraction of $\sqrt{d}$.
\newblock {\em Proc. Lond. Math. Soc. (3)}, 101(2):337--391, 2010.

\bibitem{SieveErdosKac}
A.~Granville and K.~Soundararajan.
\newblock Sieving and the {E}rd{\H o}s-{K}ac theorem.
\newblock In {\em Equidistribution in number theory, an introduction}, volume
  237 of {\em NATO Sci. Ser. II Math. Phys. Chem.}, pages 15--27. Springer,
  Dordrecht, 2007.
  
\bibitem{Hensley}
D.~Hensley.
\newblock The distribution of {$\Omega(n)$} among numbers with no large prime
  factors.
\newblock In {\em Analytic number theory and {D}iophantine problems
  ({S}tillwater, {OK}, 1984)}, volume~70 of {\em Progr. Math.}, pages 247--281.
  Birkh\"{a}user Boston, Boston, MA, 1987.

\bibitem{Hildebrand}
A.~Hildebrand.
\newblock On the number of prime factors of integers without large prime
  divisors.
\newblock {\em J. Number Theory}, 25(1):81--106, 1987.

\bibitem{IwaniecJutila}
H.~Iwaniec and M.~Jutila.
\newblock Primes in short intervals.
\newblock {\em Ark. Mat.}, 17(1):167--176, 1979.

\bibitem{Jutila}
M.~Jutila.
\newblock On mean values of {D}irichlet polynomials with real characters.
\newblock {\em Acta Arith.}, 27:191--198, 1975.

\bibitem{KP}
P.~Koymans and C.~Pagano. On the distribution of $\text{Cl}(K)[l^\infty]$ for degree $l$ cyclic fields.
\textit{arXiv preprint, 2018.}

\bibitem{KX}
E.~Knight and S.Y.~Xiao. The negative Pell equation.
\textit{arXiv preprint, 2019.}

\bibitem{Landau}
E.~Landau.
\newblock Losung des {L}ehmer'schen {P}roblems.
\newblock {\em Amer. J. Math.}, 31(1):86--102, 1909.

\bibitem{MacWilliams}
J.~MacWilliams.
\newblock Orthogonal matrices over finite fields.
\newblock {\em Amer. Math. Monthly}, 76:152--164, 1969.

\bibitem{Norris}
J.~R. Norris.
\newblock {\em Markov chains}, volume~2 of {\em Cambridge Series in Statistical
  and Probabilistic Mathematics}.
\newblock Cambridge University Press, Cambridge, 1998.
\newblock Reprint of 1997 original.

\bibitem{Selberg}
A.~Selberg.
\newblock Note on a paper by {L}. {G}. {S}athe.
\newblock {\em J. Indian Math. Soc. (N.S.)}, 18:83--87, 1954.

\bibitem{Smith8}
A. Smith. Governing fields and statistics for $4$-Selmer groups and $8$-class groups.
\textit{arXiv preprint, 2016.}

\bibitem{Smith}
A. Smith. $2^\infty$-Selmer groups, $2^\infty$-class groups, and Goldfeld's conjecture.
\textit{arXiv preprint, 2017.}

\bibitem{Stevenhagen}
P. Stevenhagen.
\newblock The number of real quadratic fields having units of negative norm.
\newblock {\em Experiment. Math.}, 2(2):121--136, 1993.

\bibitem{Redei-Stevenhagen}
P. Stevenhagen. On R\'{e}dei's biquadratic Artin symbol.
\textit{arXiv preprint, 2018.}

\bibitem{Tudesq}
C.~Tudesq.
\newblock Majoration de la loi locale de certaines fonctions additives.
\newblock {\em Arch. Math. (Basel)}, 67(6):465--472, 1996.
\end{thebibliography}
\end{document}